\documentclass[12pt,a4paper]{amsart}%
\usepackage{imakeidx}
\usepackage{amsmath}
\usepackage{fullpage}
\usepackage{comment}
\usepackage{MyCommA}
\usepackage[final]{pdfpages}
\usepackage{tikz-cd}
\makeindex

\newcommand{\Rami}[1]{{{#1}}}
\newcommand{\RamiA}[1]{{{#1}}}
\newcommand{\RamiB}[1]{{{#1}}}
\newcommand{\RamiC}[1]{{{#1}}}

\newcommand{\RamiE}[1]{{{#1}}}
\newcommand{\RamiF}[1]{{{#1}}}

\newcommand{\RamiG}[1]{{{#1}}}
\newcommand{\EitanA}[1]{{{#1}}}
\newcommand{\Dima}[1]{{{#1}}}

\newcommand{\EitanB}[1]{{{#1}}}
\newcommand{\DimaA}[1]{{{#1}}}

\newcommand{\NextVer}[1]{}
\newcommand{\sub}{\subset}
\newcommand{\noleft}{\left.\kern-\nulldelimiterspace}

\newcommand{\alp}{\alpha}

\newcommand{\gi}{\ug_i}

\begin{document}
	
	\author[Aizenbud]{Avraham Aizenbud}
	\address{Avraham Aizenbud,
		Faculty of Mathematical Sciences,
		Weizmann Institute of Science,
		76100
		Rehovot, Israel}
	\email{aizenr@gmail.com}
	\urladdr{https://www.wisdom.weizmann.ac.il/~aizenr/}
	
	\author[Gourevitch]{Dmitry Gourevitch}
	\address{Dmitry Gourevitch,
		Faculty of Mathematical Sciences,
		Weizmann Institute of Science,
		76100
		Rehovot, Israel}
	\email{dimagur@weizmann.ac.il}
	\urladdr{https://www.wisdom.weizmann.ac.il/~dimagur/}

\author[Kazhdan]{David Kazhdan}
	\address{David Kazhdan,
    Einstein Institute of Mathematics, Edmond J. Safra Campus, Givaat Ram The
Hebrew University of Jerusalem, Jerusalem, 91904, Israel}
	\email{david.kazhdan@mail.huji.ac.il}
	\urladdr{https://math.huji.ac.il/~kazhdan/}
	
\author[Sayag]{Eitan Sayag}
	\address{Eitan Sayag,
    Department of Mathematics, Ben Gurion University of the Negev, P.O.B. 653,
Be’er Sheva 84105, ISRAEL}
	\email{eitan.sayag@gmail.com}
	\urladdr{www.math.bgu.ac.il/~sayage}

	\date{\today}
	

	%
	%
	%
	%
	%
	%
	%
	%


    

\title{The jet schemes of the Nilpotent Cone of $\gl_n$ over $\F_\ell$ and analytic properties of the Chevalley map}
\maketitle
\begin{abstract} 
We prove dimension bounds on the jet schemes of the variety of nilpotent matrices (and of related varieties) in positive characteristic. 

\RamiF{This result has applications to the analytic properties of the Chevalley map $p:\fg\fl_n\to \fc$ that sends a matrix to its characteristic polynomial.
} 
We show that \RamiF{our dimension bound} implies, under the assumption of existence of resolution of singularities in positive characteristic, that the Chevalley map pushes a smooth compactly supported measure to a measure whose  density function is $L^\RamiF{t}$ for any  \RamiF{$t<\infty$.}

We also prove this analytic \RamiF{property of the Chevalley map}, unconditionally, when the characteristic of the field exceeds \RamiF{$\frac n2$}.

The zero characteristic counterpart of this result is an important step in the proof of the celebrated Harish-Chandra's integrability theorem. In a sequel work \cite{AGKS2} we show that also in positive characteristic, this analytic statement implies Harish-Chandra's integrability theorem for cuspidal representations of the general linear group. 
\end{abstract}

\tableofcontents 
\section{Introduction}

\subsection{\RamiE{Results on dimensions of jet schemes}}
Fix a finite field $\F_\ell$. Unless explicitly stated otherwise, all the algebraic varieties that we consider will be defined over $\F_\ell$. For a variety $\bfX$ we denote by $\mathcal{J}_m(\bfX)$\index{$\mathcal{J}_m$} its $m$-th jet scheme. We consider $\mathcal{J}_m$ as a functor from the category of varieties to the category of schemes. We fix an integer $n$ and set $\ug:=\gl_n$\index{$\ug$} considered as an algebraic  variety. 
In this paper we prove:
\begin{introtheorem}[\S \ref{sec:pfs}]\label[introtheorem]{thm:jet.nil}   
Let $\bfN\subset \ug$ be the nilpotent cone.
There is a constant $\Rami{C}_0$ such that for any $m \in \N$ we have
    $$\dim \mathcal{J}_m(\bfN)<m\dim(\bfN)+\Rami{C}_0.$$     
\end{introtheorem}
\RamiE{We deduce from this result bounds on jet schemes of more varieties. To formulate these bounds we make:}
\begin{notation}
Denote by
\begin{itemize}
    \item  $\uc$\index{$\uc$} - the affine space of monic polynomials of degree $n$. \RamiC{We will identify it with $\A^n$.}
\item $p:\ug\to \uc$\index{$p$} - the Chevalley map (essentially sending an element to its characteristic polynomial). 
    \item
    For an integer $i\in \N$ we denote by
    $\ug_i
    :=\underbrace{\ug\times_{\uc} \dotsc\times_{\uc} \ug}_{i \text{ times}}$\index{$\ug_i$} the $i$-folded fiber product of $\ug$ with itself over $\uc$ with respect to the map $p$.\footnote{A-priory this is a scheme, but we will see in \Cref{lem:X.i.red} below that it is reduced, so it is a variety.}
\end{itemize}
\end{notation}

We deduce the following:
\begin{introtheorem}[\S \ref{sec:pfs}]\label[introtheorem]{thm:jet.fib}   
There is a constant $\Rami{C}$ such that 
for any $x\in \uc$ and any $m \in \N$ we have
    $$\dim \mathcal{J}_m(p^{-1}(x))<m\dim(p^{-1}(x))+\Rami{C}.$$     
\end{introtheorem}
From this we deduce the following:

\begin{introtheorem}[\S \ref{sec:pfs}]\label[introtheorem]{thm:jet.fib2}   
For any $i$ there is a constant $\Rami{C}_i$ such that for any $m \in \N$ we have
    $$\dim \mathcal{J}_m(\gi)<m\dim(\gi)+\Rami{C}_i.$$     
\end{introtheorem}
\subsection{\RamiE{Results on pushforward of measures}}
\RamiE{We deduce from the results above the following one.}
\RamiB{    
\begin{introtheorem}[\Rami{\S \ref{sec:Pf.alm.an.frs}}]\label[introtheorem]{thm:alm.an.frs}  
Let $i \in \N$. Assume that 
the variety $\gi$ admits a strong resolution of singularities. 
Let $\mu^\fc$ be a Haar measure on $\fc$.

Then 
for any smooth compactly supported measure $\mu$  on $\g:=\ug(F)$, 
there exists a function $$f \in \bigcap_{t\in [1,i)} L^{t}(\fc)$$ such that  
$p_*(\mu)=f\mu^\fc.$
\end{introtheorem}
}

\begin{remark}
    In \S \ref{sec:alt.form}  we give several alternative conditions on resolution of singularities under which the result holds.
\end{remark}

Finally we \RamiE{show that one can replace the assumption of \Cref{thm:alm.an.frs} on the existence of resolution with an assumption on characteristic:}

\begin{introtheorem}[\Rami{\S \ref{sec:PfUncond}}]\label[introtheorem]{thm:uncond.an.frs}   

Suppose $\chara(\F_{\ell})>\frac{n}{2}$.  Let \RamiB{$F:=\F_{\ell}((t))$.}

Then for any smooth compactly supported measure $\mu$  on $\g:=\ug(F)$, the measure $p_*(\mu)$ can be written as a product of a  function in  $L^{\infty}(\fc)$ and a Haar measure on $\fc$.
\end{introtheorem}

\subsection{Background and motivation}
\subsubsection{FRS maps}
Theorems \ref{thm:alm.an.frs} and \ref{thm:uncond.an.frs} are related to the notion of FRS maps introduced and studied in 
\cite{AA}. Let us recall this notion:
\begin{definition}\label[definition]{def:rat}
    A map $\phi:\bfX\to\bfY$ of smooth algebraic varieties over a field of characteristic zero is called FRS\index{FRS} if it is flat, its fibers are reduced, and the singularities of its fibers are  rational. 
\end{definition}
The motivation to this definition is the following:
\begin{theorem}[{\cite[Theorem 3.4]{AA}, \cite{rei18}}]\label[theorem]{thm:AA}
    Let $\phi$ be a map of smooth algebraic varieties over a local field $F$ of characteristic zero. 
    
    If $\phi$ is FRS then for any smooth compactly supported measure $\mu$ on $\bfX(F)$, the measure
     $\phi_*(\mu)$ on $Y:=\bfY(F)$ can be written as a product of a  continuous function and a smooth measure on $Y.$
\end{theorem}
Unfortunately, we do not have an extension of this Theorem to the positive characteristic case.
In fact it is not even clear how to formulate it correctly since there is no universally accepted definition of rational singularities (see \cite{Nob, Kar,Bah, Kov} for several related notions).

For this paper we choose the following notion of rational singularities in positive characteristic.
\RamiA{
\begin{definition}\label[definition]{def:int}
Let $\bfZ$ be a variety defined over an arbitrary field. We say that the singularities of $\bfZ$ are rational  
if $\bfZ$ is Cohen-Macaulay, normal, and
 admits a resolution of singularities 
        $\eta:\tilde \bfZ\to\bfZ$ such that the natural morphism $\eta_*(\Omega_{\tilde \bfZ})\to i_*(\Omega_{\bfZ^{sm}})$ is an isomorphism. Here $i:\bfZ^{sm}\hookrightarrow \bfZ$ is the embedding of the smooth locus and $\Omega$ denotes the sheaf of top differential forms.
\end{definition}
}
\RamiA{
\begin{remark}\label[remark]{rem:rat}
In characteristic zero, this notion is equivalent to rational singularities, see {\it e.g.} \cite[\RamiE{Appendix B,} Proposition 6.2]{AA}.
\end{remark}
}
\RamiB{
Next we give several extensions of the notion of FRS maps to positive characteristic: 
}
\begin{definition}\label[definition]{def:FRS}
    Let $\phi:\bfM\to \bfN$ be a flat morphism of smooth algebraic varieties over a local field $F$ of arbitrary characteristic. Assume that the fibers of $\phi$  are reduced and normal.
    \begin{enumerate}\label{def:FRS:1}
        \item\label{def: geo.ana-FRS} We say that $\phi$ is {\bf geometrically FRS}\index{geometrically FRS, geo-FRS} (in short geo-FRS) if for any $y\in \bfN(\bar F)$, the \RamiB{singularities of}  $\phi^{-1}(y)$ \RamiA{are rational.}
        \item\label{def:FRS:2} We say that $\phi$ is {\bf analytically FRS}\index{analytically FRS, an-FRS} (in short an-FRS) if for any smooth compactly supported measure $\mu_M$ on $M:=\bfM(F)$ there exist a smooth compactly supported measure $\mu_N$ on $N:=\bfN(F)$, and a bounded function $f$ on $N$ such that $$\phi_*(\mu_M)=f\mu_N.$$      \item \label{def:FRS:3}
        We say that $\phi$ is {\bf almost analytically FRS}\index{almost analytically FRS, almost an-FRS} (in short almost an-FRS) if for any smooth compactly supported measure $\mu_M$ on $M$ 
        there exist smooth compactly supported measure $\mu_N$ and a function $f$ on $N$ 
        \RamiB{such that
        $$\phi_*(\mu_M)=f\mu_N.$$    and  $f \in L^r(N)$ for all $r\in [1,\infty)$}
     \end{enumerate}
     Using extension of scalars from $\F_\ell$ to $\F_\ell((t))$ we will apply these notions also for maps of varieties over $\F_\ell$.
\end{definition}
\begin{remark}
\RamiB{As in \Cref{rem:rat}, in characteristic zero, the geo-FRS property is equivalent to FRS property. Also, by \Cref{thm:AA}, in this case each of them implies the an-FRS property.}
\end{remark}
{In this language, 
\Rami{\Cref{thm:alm.an.frs} implies that, under an appropriate assumption of existence of resolution of singularities, the Chevalley map, $p:\fg \to \fc$, is almost an-FRS. Similarly,}
the content of \Cref{thm:uncond.an.frs}   is that, for \RamiB{$\chara(\F_{\ell})>\frac{n}{2}$}
the Chevalley map, $p:\fg \to \fc$, is an-FRS.}

The following is a positive characteristic analogue of \Cref{thm:AA}:
\begin{introconj}\label{conj:geo.imp.an}

Let $\phi:\bfM\to \bfN$ be 
{a flat morphism of smooth algebraic varieties \RamiB{defined over a local field $F$} whose fibers are reduced and normal.} Assume that $\phi$ is geo-FRS. Then $\phi$ is an-FRS.
\end{introconj}
We explain \RamiA{in \S \ref{sssec:springer}} below that this would imply the following:
\begin{introconj}\label{conj:chev.is.an.frs}
The Chevalley map $p:\ug\to \uc$ is \RamiA{an-}FRS. 
\end{introconj}
Theorems \ref{thm:alm.an.frs} and  \ref{thm:uncond.an.frs} are partial results towards this conjecture.

\subsubsection{The Springer resolution}\label{sssec:springer}
The nilpotent cone has a natural resolution \Rami{of singularities}
$$T^{*}(\fB) \to \bfN$$ 
by the cotangent bundle to the flag variety, called the Springer resolution. In characteristic zero, one can use this resolution in order to prove that the singularities of the nilpotent cone are rational (see \cite[Theorem A]{Hes}). We list now two important corollaries of this fact: 
\begin{enumerate}[(I)]
    \item\label{it:jet.nil} The jet schemes of the nilpotent cone are irreducible of the expected dimensions (See \cite[Appendix]{Mu}). 
    \item\label{it:chev.an.frs} The Chevalley map $p:\ug\to \uc$ is an-FRS. This follows from the rationality of the singularities of the nilpotent cone using \cite[Theorem 3.4]{AA}. This fact is essentially equivalent to the bounds on orbital integrals established in \cite[Theorem 13, page 68]{HC_VD}. 
    Note that \cite[Theorem 13]{HC_VD} is an important ingredient in Harish-Chandra's integrability theorem for characters of irreducible representations of reductive $p$-adic groups (See \cite{HC_int}). We will discuss this below in more details \RamiB{(see \S \ref{Subsec: HC integ.}).}
\end{enumerate}
\EitanA{
In general positive characteristic, these results \Rami{are not known}.\footnote{ \EitanA{When $\chara(F)>n$, the  jet scheme of the nilpotent cone is known to be of the expected dimension, see \cite[Theorem 3.4.7]{BoKV}.}}}

\RamiB{
In more details, the Springer resolution exists in any characteristic, moreover, according to our definition of rational singularities (\Cref{def:rat}) one can deduce from it that the singularities of the fibers of the Chevalley map are rational (it follows from the proof in \cite{Hin} of \cite[Theorem A]{Hes}). So we have:
\begin{prop}\label{thm:chev.is.frs}
The Chevalley map $p:\ug\to \uc$ is \RamiA{geo-}FRS. \footnote{\RamiA{Indeed, by Jordan-Chevalley decomposition it is enough to show that the Springer resolution satisfies the condition of \Cref{def:int}. This is proven in \cite{Hin} for a generalization of the Springer resolution. The statements in  \cite{Hin} are formulated for characteristic zero, but the proof of this  statement is valid in arbitrary characteristic.
}}
\end{prop}
However,
}
the main results of \EitanA{\cite{AA,Mu} and \cite[Theorem 13]{HC_VD}} are not known in positive characteristic for the following reasons:
\begin{enumerate}
    \item \cite{AA,Mu} use \RamiB{repeatedly} 
    the existence of resolution of singularities, which is not known in positive characteristic.
    \item \cite{AA,Mu} use the Grauert-Riemenschneider theorem,  which is not valid (in general) \RamiB{over fields of} positive characteristic (see \cite{Ray} and 
    \cite{Meh}).
    \item The proof of \cite[Theorem 13]{HC_VD} uses the fact that the Lie algebra of $G$ is the direct sum of its center and its derived algebra. This is not true in positive characteristic. 
    \item \RamiB{The proof of \cite[Theorem 13]{HC_VD} uses the Jordan-Chevalley decomposition. This decomposition does not exist as is over local fields of positive characteristic, since there are elements with irreducible totally inseparable characteristic polynomial.}    
\end{enumerate}

We view Theorems \ref{thm:jet.nil}   - \ref{thm:uncond.an.frs}   as partial analogs of \eqref{it:jet.nil} and \eqref{it:chev.an.frs}, \RamiA{and as evidence for Conjectures \ref{conj:geo.imp.an} and \ref{conj:chev.is.an.frs}.}

\subsubsection{Harish-Chandra's integrability theorem}\label{Subsec: HC integ.}
Another motivation for Theorems \ref{thm:jet.fib2} and  \ref{thm:alm.an.frs} is the following celebrated result by Harish-Chandra.
\begin{theorem}[\cite{HC_int}]\label[theorem]{thm:HC.int}

    Let $F$ be a local non-Archimedean field with {\bf characteristic zero.} Let $\bfH$ be a reductive group defined over $F$. 
 Let  $H=\bfH(F)$ and fix a Haar measure $\mu$ on $H$.
    Let $(\rho,V)$ be an irreducible smooth representation of $H$. Consider the distributional character $\chi_{\rho}$ of $\rho$ defined by 
$
  \chi_{\rho}(f)=trace(\rho(f\mu))$ where 
$$  \rho(f\mu)(v)=\int_{H} f(g)\rho(g)v \mu$$
   and $f \in C_c^{\infty}(G), v\in V$ .
    
    Then $\chi_{\rho}$ is represented by a locally integrable function on $H$.
\end{theorem}
An important step in the proof of \Cref{thm:HC.int} is a bound on orbital integrals \cite[Theorem 13]{HC_VD}. This bound is essentially equivalent to the fact that the Chevalley map is an-FRS. Thus we view 
\Cref{thm:alm.an.frs} 
as a partial positive characteristic analogue of  \cite[Theorem 13]{HC_VD}.

In \cite{AGKS2} we will use Theorems \ref{thm:alm.an.frs} and  \ref{thm:uncond.an.frs} in order to prove a positive characteristic version of Harish-Chandra local integrability theorem (\Cref{thm:HC.int}) for irreducible cuspidal representations of $\GL_n$ under the assumption of the existence of resolutions of singularities as in \Cref{thm:alm.an.frs} or the assumption on characteristic of \Cref{thm:uncond.an.frs}.

\RamiB{The \RamiE{main statement of \cite{Le4}} is \Cref{thm:HC.int} for $\GL_n$ in positive characteristic. However, this proof contains a gap that we explain in details in 
\cite{AGKS2}. 
Moreover, t}he statement of \cite[\S5.3 Corollary 1]{Le1}  is equivalent to the statement of \Cref{thm:uncond.an.frs} but without limitation on characteristic. However,  the proof of  \cite[\S5.3 Corollary 1]{Le1}  
contains a gap. 
Namely, it is based on the lemma in \cite[\S3.7]{Le1} which is wrong as stated. This mistake is a propagation of \RamiC{an earlier mistake} \RamiB{from \cite[Lemma 5.4.2]{Le5}}, \RamiE{which in turn comes from a mistake in \cite[Lemma 2.3.2]{Le4}.}
\subsection{Further related results}
One can use the methods of \cite{CGH} to prove that Theorems \ref{thm:jet.nil}, \ref{thm:jet.fib}, and \ref{thm:alm.an.frs} are valid for large enough characteristic of $\F_\ell$.
However, no explicit bound on the characteristic can be obtained in this way.

\subsection{Ideas of the proofs}
\subsubsection{The original Harish-Chandra's argument}
The starting point for this paper is the original Harish-Chandra's proof of the
bound on the orbital integrals of a compactly supported function on $\fg\fl_n$ in the characteristic zero case (See \cite[Theorem 13]{HC_VD}). This can be reformulated as the statement that
the Chevalley map is an-FRS. Let us briefly recall Harish-Chandra's argument in the language of an-FRS maps:
\begin{enumerate}
    \item  Decompose $\fg\fl_n$ as  $\fs\fl_n \oplus \fg\fl_1$ and reduce the statement to the statement that  the Chevalley map $p':\fs\fl_n\to \fc'$ is an-FRS.
    \item Deduce from the induction hypothesis that the  map  $p'|_{\fs\fl_n\smallsetminus N}$ is an-FRS, where $N$ is the nilpotent cone.
    \item Prove by descending induction that $p'|_{\fs\fl_n\smallsetminus S}$ is an-FRS, where \RamiC{$S\subset N$} is a  closed $\GL_n$-invariant subset. The induction is on $S$:
\begin{enumerate}
\item For each nilpotent $x$ use the Slodowy slice $L_x$ at $x$ to reduce the statement that 
$p'|_{L_x}$ is an-FRS \RamiB{to the statement that} $p'|_{L_x\smallsetminus 0}$ is an-FRS.
\item Use the action of $\bG_m$ on $L_x$ and an estimate on its eigenvalues in order to prove that  $p'|_{L_x}$ is an-FRS assuming the fact that $p'|_{L_x\smallsetminus 0}$ is an-FRS.
\end{enumerate}
\end{enumerate}
When going to positive characteristic this argument has several problems:
\begin{itemize}
    \item Step (1) is invalid.
    \item Step (2) is invalid as stated, but one can adapt it to prove that $p|_{\fs\fl_n\smallsetminus N_{insep}}$ is an-FRS where 
    $N_{insep}$ is the collection of matrices with  purely inseparable characteristic polynomial.\footnote{\RamiB{For example the matrix $\begin{pmatrix}
        0 &1\\ t &0\\ 
    \end{pmatrix}$ over the field $\F_2((t))$.}}
    \item Step (3) is invalid as stated, but one can replace the Slodowy slice (\cite{Slo}) with other constructions (see \S\ref{sec:slices}).   \RamiB{However, this method can be applied directly only to the nilpotent cone $N$ and  not to $N_{insep}$.}
\end{itemize}
\subsubsection{An-FRS over the origin}
One can start with proving a weaker statement. Namely, that $p$ is an-FRS over the origin. This means that for every smooth compactly supported measure $\mu$, the density of  $p_*(\mu)$ at \RamiC{$0\in \fc$} is finite. Note that this does not \RamiB{imply}
that this density is bounded in any neighborhood \RamiC{of the origin}. In fact, in order to enable Harish-Chandra's argument, we need a very narrow definition of density at a point - the limit of the average density for a very specific sequence of balls that converges to $0$.

For the weaker statement, steps (1) and (2) become obsolete, since it is obvious that $p_{\g\smallsetminus N}$ is an-FRS over the origin (as \RamiC{the origin} is not in its image). One can adapt step (3) to work in this case.
\subsubsection{Effectively an-FRS over the origin}

The property of being an-FRS \RamiC{over the origin} by itself is not useful for us, since we can not use it in order to deduce any information outside \RamiC{the origin}. In order to make it useful we have to consider \RamiB{a}
version which is uniform over finite extensions of $\F_\ell$. This leads us to the notion of effectively an-FRS over the origin (see \Cref{def:an.FRS.0} below). \RamiB{Our} proofs of Theorems \ref{thm:jet.fib2} and \ref{thm:alm.an.frs} are based on the proof that $p$ is effectively an-FRS over the origin (see \Cref{thm:an.jet} below). The proof follows the lines described above, but with several \RamiC{important adaptations:}
\begin{itemize}
    \item One has to properly define a setting where one can state effective results that are uniform on finite extensions of $\F_\ell$. This we did in \cite{AGKS0}.
    \item One has to reprove standard facts from differential geometry (such as the implicit function theorem) in this effective setting. This \RamiB{was also done} in \cite{AGKS0}.
    \item One has to carefully define the notion of an-FRS maps in a way that makes it \RamiB{amenable} for Harsh-Chandra's argument. 
    The key points here are to work with a specific sequence of ellipsoids around \RamiC{the origin} and to require the bound on the average density to be uniform \RamiB{over the entire sequence (}even for large \RamiB{ellipsoids)}. This makes the fact that $p|_{\fg\fl_n\smallsetminus N}$ is effectively an-FRS over the origin not obvious, but still correct (see \Cref{lem:not.in.im} below).
\item The Harish-Chandra's argument uses (implicitly) the fact that the notion of an-FRS map is local with respect to smooth covers which are surjective on the level of $F$ points. This locality is not clear for maps that are effectively an-FRS over the origin. However, we show that this notion is local with respect to a class of smooth covers which we call effectively surjective (see Definition  \ref{def:effsur} below). \Rami{In \cite{AGKS0} we proved} that this class of covers includes the Nisnevich covers (See  \Cref{prop:crit-effsurj} below). This allows us to adapt Harish-Chandra's argument.
\end{itemize}

\subsubsection{Proof of Theorems \ref{thm:jet.nil}, \ref{thm:jet.fib},     and \ref{thm:jet.fib2}}
We deduce \Cref{thm:jet.nil} from the fact that $p$ is effectively an-FRS over the origin using the Lang-Weil bounds. If the ellipsoids in the definition of effectively an-FRS over the origin would be balls, this would be straightforward - the average density would be exactly the limit of the normalized number of points in the jet-scheme. In our case, the bounds on the average density provide bounds on the dimension of some weighted versions of the jets schemes. We deduce the desired bounds on dimensions from this bound using the semi-continuity of dimension \RamiB{of fibers}. This is done in \S \ref{sec:jet} (\Cref{thm:jet.nil} itself is proven in \S \ref{sec:pfs} but the actual \RamiB{argument is in \S \ref{sec:jet}}.)

We deduce \Cref{thm:jet.fib} from \Cref{thm:jet.nil} using semi-continuity of dimension \Rami{again}. \Cref{thm:jet.fib2} follows easily from \Cref{thm:jet.fib}. See \S\ref{sec:pfs}.
\subsubsection{Proof of \Cref{thm:alm.an.frs}} We now want to go back to an-FRS property but this time over the entire range.  This we do only under the additional assumption of existence  of resolution of singularities.

For  a variety $\bfX$ (equipped with a top form on its  smooth locus) we can consider the following quantity:
the volume of a ball of radius $R$  in $\bfX^{sm}(\F_{\ell^k}((t)))$. Here $\bfX^{sm}$ denotes the smooth locus of $\bfX$. Note that the notion of ball in $\bfX^{sm}$ is defined in such a way that the singular locus of $\bfX$ is infinitely far. 
We can study the asymptotics of this volume both as a function of  $R$ and as a function of $k$. \RamiB{We note that} our analysis is local on $\bfX$ so all the balls are intersected with a fixed ball in $\bfX$. 

Under our assumption of existence of resolution of singularities we relate these 2 asymptotics. See \Cref{thm:eq.almost.int}.

We apply \Cref{thm:eq.almost.int} to the variety $\ug_i$.  We relate the 
asymptotic behavior of the above volume 
with respect to $k$ to the number of 
points on the jet schemes of 
$\ug_i$. See \Cref{thm:alm.int.fib}. We 
use \Cref{thm:eq.almost.int} to deduce bounds 
on the asymptotics of the above 
volume with respect to $R$.

Now we wish to relate this asymptotics to the \RamiB{desired} $L^i$ property. 
We fix a measure $\mu$ on $\ug(\F_\ell((t)))$ and consider its pushforward under the Chevalley map $p:\ug(\F_\ell((t)))\to\uc(\F_\ell((t)))$.  Denote this resulting measure by $\nu$. We consider the $L^i$  norm of $\nu$ outside an $\eps$-neighborhood of 
the singular locus of $p$. 

We relate \Rami{the asymptotics} of this $L^i$  norm w.r.t. $\eps$ to the above asymptotics  w.r.t. $R$.  We use \Rami{a} standard \Rami{analytic} argument to deduce from this a bound on the 
$L^{i'}$-norm for any $i'<i$. See 
\S \ref{sec:Pf.alm.an.frs}.
\subsubsection{Proof of \Cref{thm:uncond.an.frs}} 
We follow the original Harish-Chandra's argument. The first step is to analyze the situation near (non-central) semi-simple elements, using the induction hypothesis. 

In the characteristic zero situation this proves the result outside the Minkowski sum of the  nilpotent cone and the set of scalar matrices. 

In the \Rami{positive characteristic} situation this proves the result outside the cone of matrices whose characteristic polynomial is purely inseparable.

In general, this cone is rather complicated, however, under our assumption on the characteristic, only its regular part exhibits this complexity. Since the statement is obvious for regular matrices, we can ignore \Rami{this complexity}, and again deduce the result outside the Minkowski sum of the (non-regular) nilpotent cone and the set of scalar matrices. 

This still does not finish the problem, since we do not have the splitting of $\ug$ to scalar and traceless matrices. So we use our assumption on the characteristic again in order to choose \Rami{an  analog of the} Slodowy slice (to a non-regular orbit) which will have such a splitting. See \Cref{lem:slod2.slice}. 

\Rami{Though the slice has a splitting, $\fc$ does not admit a corresponding splitting. Thus the original Harish-Chandra's homogeneity argument does not work here. One has to tweak it in order to take into account the one-dimensional center (see \Cref{lem:FRS.hom.crit2}). This makes it less sharp, so it fails to give the an-FRS property around the subregular orbit (though it still gives the almost an-FRS property). Therefore we prove the an-FRS property around the subregular orbit using a direct computation - see Step \ref{step:4.pf.E} of the proof of \Cref{thm:uncond.an.frs}.}
\RamiE{
\subsubsection{The role of the assumption $\bfG=\GL_n$ }
We used the assumption $\bfG=\GL_n$ in order to make all explicit computations easier. However, our argument does not use any statement that inherently depends on this assumption (such as existence of mirabolic subgroup, stability of adjoint orbits, or the Richardson property of all nilpotent orbits).

Moreover, for non-type $A$ groups in good characteristic the analog of $N_{insep}$ coincides with the nilpotent cone $N$. Therefore, we expect the original Harish-Chandra's argument to allow reduction to type $A$. Hence, we expect that in good characteristic\footnote{see e.g. 
\cite[I, §4]{SS} for the definition of this notion}, the conclusion of \Cref{thm:alm.an.frs} for general reductive group will only require the assumptions on resolution of the variety $\ug_i$  for $\ug=\fg\fl_n$. 
}
\subsection{Structure of the paper}
In \S \ref{sec:prel} we fix some conventions and recall some standard facts.

In \S \ref{sec:norm} we give a short overview of the theory of norms on algebraic varieties over local fields developed in \cite[\S18]{Kot}.

In \S \ref{sec:rect}, we recall the theory of rectified algebraic varieties and balls and measures on them, introduced in \cite{AGKS0}. 

We recall the main results of \cite{AGKS0}, which are uniform analogs of standard results from local differential topology, including the implicit function theorem and study the behavior of smooth measures under push forward {with respect to} submersions.

We also introduce the notion of effectively surjective map from \cite{AGKS0}, which is  a surjective map such that we can control the \RamiA{norm of a} preimage in terms of \RamiA{the norm} of the point in the target. We recall a statement  
from \cite{AGKS0} that implies that any Nisnevich cover is effectively surjective.

In \S \ref{sec:basic}  we recall some standard facts on the Chevalley map that are less standard in positive characteristic.

In \S \ref{sec:EfanFRS0}
we introduce the notion of effectively an-FRS over the origin and prove that the Chevalley map has this property.

In \S \ref{sec:jet} 
we prove a bound on the dimension of the jets of the fiber of an  effectively an-FRS map over \RamiA{the origin}.

In \S \ref{sec:pfs}
we deduce Theorems \ref{thm:jet.nil}, \ref{thm:jet.fib} and \ref{thm:jet.fib2}. 

In \S \ref{sec:AlmInt}
we introduce several notions related to the asymptotics of the volumes of balls \RamiC{in the} smooth loci of varieties over local function fields. We call these notions \Rami{analytic, geometric and asymptotic} almost integrability. 

We show that all these notions are equivalent under the assumption of existence of an appropriate resolution of singularities.

In \S \ref{sec:AlmIntgi} we prove that the varieties $\ug_i$ are \Rami{asymptotically} almost integrable.

In \S \ref{sec:Pf.alm.an.frs}
we prove \Cref{thm:alm.an.frs}.

In \S \ref{sec:alt.form}
   we give several \RamiB{variations} of \Cref{thm:alm.an.frs}.

In \S\ref{sec:PfUncond}
we prove \Cref{thm:uncond.an.frs}.
\subsection{Acknowledgments}
We would like to thank Dan Abramovich and Michael Temkin for enlightening conversations about resolution of singularities. 
We would also like to thank Nir Avni for many conversations on algebro geometric analysis. \EitanB{We would like to thank Alexis Bouthier and Yakov Varshavsky for helpful discussions on the flatness of Chevalley map.}

During the preparation of this paper, A.A., D.G. and E.S. were partially supported by the ISF grant no. 1781/23. 
D.K. was partially supported by an ERC grant 101142781.

\section{Notations and Preliminaries}\label{sec:prel}
\subsection{Conventions}\label{ssec:conv}
\begin{enumerate}
    \item By a variety we mean a reduced scheme of finite type over a field.
    Unless stated otherwise this field will be $\F_\ell$.\index{variety}
        \item When we consider a fiber product of varieties, \Rami{and fibers of maps between varieties,} we always consider it in the category of schemes.
    \item We will describe \Rami{subschemes} and morphisms of varieties \Rami{and schemes}  using set-theoretical language, when no ambiguity is possible.    
    \item We will usually denote algebraic varieties by bold face letters (such as $\bfX$).
    \item For Gothic letters we use underline instead of boldface.
    \item We will use the same letter to denote a morphism between algebraic varieties and the corresponding map between the sets of their $F$-points for various  fields $F$.
    \item We will use the symbol $\square$\index{$\square$} in \Rami{the} middle of a square diagram in order to denote that a square is Cartesian. 
    \item A big open set of an algebraic variety $\bfZ$  is an open set whose complement is of co-dimension at least 2 (in each component)\index{big open set}
\item For an algebraic variety $\bfX$, we denote by $\bfX^{sm}$ the variety of smooth points of $\bfX.$ We also denote by $\bfX^{sing}$ the variety of singular points. \index{$\bfX^{sm}$}\index{$\bfX^{sing}$}
\item We will abbreviate SNC divisor for strict normal crossings divisor.\index{SNC divisor}
    \item\label{ssec:conv:strong} By a strong resolution of singularities\index{strong resolution of singularities} of a  variety $\bfX$ we mean a \Rami{proper birational map}   $\phi:\tilde \bfX\to \bfX$ s.t. \Rami{$\tilde \bfX$ is smooth, $\phi$} is an isomorphism over $\bfX^{sm}$ and the inverse image of  $\bfX^{sing}$ in $\tilde \bfX$, considered as a variety, is an SNC divisor.
    \item\label{not.l.les.inf} If $F$ is a local field and $X$ is an analytic manifold over $F$ (In the sense of \cite[Part II, Chapter 3]{Ser_Lie}) and $r>1$ we denote by $L^{r}_{\loc}(X)$  the space of functions on $X$ which are locally in $L^r$. That is, functions $f$ s.t. for any open analytic embedding $\phi:U\to X$, with  $U\subset F^M$  precompact, we have $f\circ \phi\in L_{\loc}^r(U)$.
    We also define $$L_{\loc}^{<\infty}(X)=\bigcup_{r} L_{\loc}^{r}(X).$$\index{$L_{\loc}^{r}$,$L_{\loc}^{<\infty}$}
    \item When considering elements in these function spaces, we will not distinguish between functions on $X$ and  functions defined \Rami{almost everywhere} in $X$.
\end{enumerate}

\subsection{Forms and measures}\label{ssec:forms}
\begin{definition}
 For a top form $\omega$ on a smooth algebraic variety $\bfX$ defined over a local field $F$,  we denote the corresponding measure on $X:=\bfX(F)$ by $|\omega|$.\index{$\vert\omega\vert$}
\end{definition}

\begin{notation}
    For a smooth morphism $\gamma:\bfZ_1\to \bfZ_2$, a top differential form $\omega_{\bfZ_2}$ on $\bfZ_2$, and a relative top differential form $\omega_{\gamma}$ on $\bfZ_1$ with respect to $\gamma$, denote the corresponding top differential form on $\bfZ_1$ by $\omega_{\bfZ_2}*\omega_{\gamma}$. 

    We use the same notation for rational top-forms. Also in this case, we do not have to require that $\bfZ_i$ and $\gamma$ are smooth, instead it is enough to require that $\gamma$ is generically smooth.\index{$*$}
\end{notation}

\begin{definition}
    Given a Cartesian square of smooth morphisms and smooth varieties:
    $$
    \begin{tikzcd}
\bfV \arrow[r] \arrow[d] \arrow[dr, phantom, "\square"] & \bfZ_1 \arrow[d] \\
\bfZ_2  \arrow[r] & \bfZ
\end{tikzcd}
$$
and top forms $\omega,\omega_i$ on $\bfZ,\bfZ_i$ define a form $\omega_1\times_{\omega} \omega_2$ on  $\bfV$ in the following way:
\begin{itemize}
    \item Let $\omega_i'$ be a Gelfand-Leray relative form on $\bf Z_i$ w.r.t. the map $\bfZ_i\to \bfZ$.
    \item $\omega_1'\boxtimes_\bfZ \omega_2'$ is the corresponding relative form on $\bfV$ w.r.t. the map $\gamma:\bfV\to \bfZ$.
    \item $\omega_1\times_{\omega} \omega_2:=\omega*(\omega_1'\boxtimes_\bfZ \omega_2')$. \index{$\times_{\omega}$}
\end{itemize}
    We use the same notation for rational 
    top-forms. Also in this case, we do not have to require that $\bfZ_i$, $\bfZ$ and $\gamma$ be smooth, instead it is enough to require that \Rami{they are} generically smooth.
\end{definition}

\section{Norms}\label{sec:norm}
In this section, we recall basic parts of the theory of norms developed in \cite[\S18]{Kot}, and prove an integrability result about  this theory (See \Cref{prop:log.norm.in.L1} below).
We will use the following notions from \cite[\S18]{Kot}.
\begin{enumerate}
    \item An abstract norm on a set $Z$ is a positive real-valued
function $||\cdot ||_Z$ on $Z$ such that $||x||_Z\geq 1$ for all $x \in Z$. \index{abstract norm}

\item For two abstract norms $||\cdot||_Z^1,||\cdot||_Z^2$ on $Z$ we say that $||x||_Z^1\prec||x||_Z^2$ if there is a constant $c>1$ s.t. $||x||_Z^1<c(||x||_Z^2)^c$.\index{$\prec$}

    \item We say that two abstract norms $||\cdot||_Z^1,||\cdot||_Z^2$ on $Z$ are equivalent, and denote this as 
     $||\cdot||_Z^1 \sim ||\cdot||_Z^2$\index{$\sim$},
     if  $||x||_Z^1\prec||x||_Z^2\prec||x||_Z^1$.

    \item Let $\bfM$ be an algebraic variety defined over  a local field $F$. In \cite[\S18]{Kot} there is a definition of a canonical equivalence class of abstract norms in $M=\bfM(F)$. The abstract norms in this class are called norms on $M$.  \index{norm}  
\end{enumerate}
The main result of this section is the following:
\begin{proposition}\label[proposition]{prop:log.norm.in.L1}
    Let $\bfU\subset \bfX$ be an open dense subset of a smooth variety. Let $X=\bfX(\F_\ell((t)))$ and $U=\bfU(\F_\ell((t)))$. Let $||\cdot||_{U}$ be a  norm on $U$. Then  $$\log_\ell^{}\circ ||\cdot||_{U}\in L^{<\infty}_{\loc}(X).$$ 
\end{proposition}
For the proof we will need some preparations.
\begin{lem}[cf. {\cite[Theorem 1.3]{GH25}}]
    Let $\bfX$ be a smooth algebraic variety defined over a local field $F$. 
    Let $f\in \cO_\bfX(\bfX)$ be a non-zero divisor. Let $X=\bfX(F)$. Then there exist $\eps>0$ s.t. $|f|^{-\eps}\in L^1_{\loc}(X)$.
\end{lem}
\begin{cor}\label[cor]{cor:log.poly.is.in.L1}
        Let $\bfX$ be a smooth algebraic variety defined over a local field $F$. 
    Let $f\in \cO_\bfX(\bfX)$. Let $X=\bfX(F)$. Then  $\log(|f|)\in L^{<\infty}_{\loc}(X)$.
\end{cor}
\begin{proof}
    WLOG assume that we have an invertible to form $\omega$ on $\bfX$. Fix a compact $C\subset X$.
    Let $$C_i:=\{x\in C: i\leq |f(x)|^{-1}<i+1\}$$
    Let $m_i:=|\omega|(C_i)$. Let $\eps$ be as in the lemma. Then we have $$\sum_{i=1}^\infty i^{\eps} m_i<\infty.$$
    So, for any $k>0$ we obtain 
    $$\sum_{i=1}^\infty (\log(i+1))^{k} m_i<\infty.$$
    This implies the assertion.
\end{proof}
\begin{proof}[Proof of \Cref{prop:log.norm.in.L1}]
    Note that this statement does not depend on the norm so we will choose the norm at our convenience.
    \begin{enumerate}[{Case} 1.]
   \item $\bfX$ is affine and $\bfU=\bfX_f\subset \bfX$ is principal open affine subset:\\
    Take any norm $||\cdot||_\bfX$ on $X$ and take 
    $$||x||_{\bfU}:=\max(||x||_\bfX, |f(x)|^{-1})$$
    The assertion follows now from \Cref{cor:log.poly.is.in.L1}.
        \item $\bfX$ is affine:\\
    Follows from the previous case.
        \item $\bfX$ is general case:\\
    Follows from the previous case.
    \end{enumerate}
\end{proof}

\section{Rectified algebraic varieties}\label{sec:rect}
In this section, we recall the theory of rectified algebraic varieties and balls and measures on them, introduced in \cite{AGKS0}. 
This is a framework for quantitative statements on distances and measures when studying algebraic varieties and morphisms of algebraic varieties over local fields of the type $\F_\ell((t))$. 
It allows to formulate uniform statements  with respect to finite extensions of $\F_\ell$.

We recall the main results of \cite{AGKS0} which are uniform analogues of standard results from local differential topology, including the implicit function theorem and study the behavior of smooth measures under push forward {with respect to} submersions.

Part of this theory is analogous to the theory of norms recalled above.

We also introduce the notion of effectively surjective map \Rami{from} \cite{AGKS0}, which is a surjective map such that we can control \RamiA{norm of a} preimage in terms of the \RamiA{norm} of the point in the target. We recall a statement  
from \cite{AGKS0} that implies that any Nisnevich cover is effectively surjective. \RamiA{See \Cref{prop:crit-effsurj}.}
\subsection{Notions}\label{ssec:notions}
We recall here the  main notions from \cite{AGKS0}:
\begin{definition}
    Let $\bfX$ be a smooth algebraic variety over $\F_\ell$.  
    
    \begin{enumerate}
        \item 
    A rectification\index{rectification, rectified variety}   of $\bfX$ is a finite open cover $\bfX \subset \bigcup_{\alpha \in I} \bfU_{\alpha}$ 
    with 
    closed embeddings $i_\alpha:\bfU_{\alpha}\to \bA^M$.    
    \item We will call a rectification simple if $|I|=1$.
    \item     By a rectified variety we will mean a smooth algebraic variety over $\F_\ell$ equipped with a rectification. \RamiA{By a map or a morphism of such we just mean a morphism of the underlying algebraic varieties.}
    \item     \RamiA{A $\mu$-rectification} of  $\bfX$  is a rectification of $\bfX$  together with invertible top differential forms $\omega_{\alpha}\in \Omega^{top}(\bfU_{\alpha})$.
    \item We define similarly the notion of a \RamiA{$\mu$-rectified} variety, and 
     simple \RamiA{$\mu$-rectification}.
    \end{enumerate}
\end{definition}

\RamiG{
\begin{definition}\label[definition]{def:simp balls}
$ $
    Let $(\bfX,i)$ be a simply rectified variety. Then, for any $m\in \Z$ define:  \begin{enumerate}
    \item \label{def:simp balls:1}  
${B_m^{\bfX}}:=i^{-1} \left(\varpi^{-m}\cO_F^{M}\right){\subset \bfX(F)}.$
        \item\label{def:simp balls:2} For $x\in \bfX(F)$ define {a ball around $x$:}
        $${B_m^{\bfX}(x)}:= i^{-1} \left(i(x)+\varpi^{-m}\cO_F^{M}\right).$$
\end{enumerate}
\end{definition}
}
\RamiG{

\begin{lem}[{\cite[Lemma 3.5]{AGKS0}}]\label{lem:KL}
Let $(\bfX,\bfU_\alpha,i_\alpha)_{\alpha\in I}$ be a rectified variety. Then there
is $s\in\N$ with the following property. For every  $\alpha,\beta\in I$,
every $m\in\N_{>0}$, every
\[
x\in B^{\bfU_\alpha}_{m}\cap B^{\bfU_\beta}_{m},
\]
and every integer $r\ge1$ 
we have 
$$B_{-r-sm}^{\bfU_\alpha}(x)\subset  B_{-r}^{\bfU_\beta}(x).$$
\end{lem}
}

\RamiG{
We choose for every rectified variety an $s$ as in the above lemma and denote it by ${s_\bfX}$. We choose it in a way that if the rectification is simple then $s=0$.

We are now ready to define balls and measures on our rectified varieties.}
\RamiG{
\begin{definition}\label[definition]{def: balls and measures}
$ $
\begin{enumerate}
    \item \label{def: balls and measures:1}
    Let $(\bfX,\bfU_\alpha,i_\alpha)$ be a rectified variety. Then, for any $m\in \Z$ define:    
            \begin{enumerate}
            \item\label{def: balls and measures:a}
${B_m^{\bfX}}:=\bigcup_{\alpha} B_m^{\bfU_\alpha}{\subset \bfX(F)}.$
        \item\label{def: balls and measures:c} For $x\in \bfX(F)$ define {a ball around $x$:}
        $${B_m^{\bfX}(x)}:= \bigcup_{\alpha, a \text{ s.t. } x\in B_{a}^{\bfU_\alpha}; a>0} B_{m-as_\bfX  }^{\bfU_\alpha}(x) .$$
        \item    For $\bfZ\subset \bfX$  we define:  $${B_m^{\bfX}(\bfZ)}:=\bigcup_{z\in {B^{\bfZ}_\infty}}   B_m^{\bfX}(z).$$ 
    \end{enumerate}
        
        \item 
    Let $(\bfX,\bfU_\alpha,i_\alpha,\omega_\alpha)$ be a \RamiB{$\mu$-rectified} variety. Then, for any {integer $m\in \Z$}  we define
        a measure $\mu_m^{\bfX}$ on $\bfX(F)$ supported on $B_m^{\bfX}$ by: $${\mu_m^{\bfX}}:=\sum_\alpha |(\omega_\alpha)_{F}| \cdot 1_{i_\alpha^{-1} \left(\varpi^{-m}\cO_F^{M}\right)}.$$
\end{enumerate}
\end{definition}
}

\begin{definition}
    By an almost affine space we mean a principal open subset in an affine space defined over $\F_\ell$. Note that  any almost affine space is equipped with a {natural} simple (\RamiA{$\mu$-})rectification {that we will call the standard  (\RamiA{$\mu$-})rectification on this space}. 

    When dealing with such space, if we are not fixing a \RamiA{$\mu$-rectification} on it, the above notions of balls and measures will refer to the standard rectification.
\end{definition}

\begin{definition}
    Let $\bfX$ be a rectified variety. Let $m,k\in\N$. We say that 
    $f\in C^\infty(B_\infty^{\bfX,k})$
    is $m$-smooth \Rami{if} for any 
    $x\in B_\infty^{\bfX,k}$
    the function $f|_{B_{-m}^{\bfX,k}(x)}$ 
    is constant.\index{$m$-smooth}
\end{definition}

\begin{defn} \label[defn]{def:effsur}
Let $\gamma:\bfX\to \bfY$ be a map of rectified varieties. We say that $\gamma$ is effectively surjective\index{effectively surjective} iff  for any $m \Rami{\in \N}$ there is $\Rami{m'} \in \N $ s.t. for every $k\in \N$ we have $$\gamma(B^{\bfX,k}_{m'})\supset B^{\bfY,k}_m.$$
\end{defn}

\subsection{Statements}
We recall here the  main statements of \cite{AGKS0}.

The following is obvious:
\begin{lem}[{\cite[Lemma 3.\RamiG{8}]{AGKS0}}]\label{lem:tri}
    Let $\bfX$ be a rectified variety. Then for any 2 \DimaA{positive} integers $m_1,m_2\in \N$ we have:
    \begin{enumerate}
        \item\label{lem:tri:mon} If $x\in B^{\bfX,k}_{m_2}$ then $B_{-m_1}^{\bfX,k}(x)\sub B^{\bfX,k}_{m_2}$.
        \item\label{lem:tri:sym} If $x\in B^{\bfX,k}_{\infty}$ and $y\in B^{\bfX,k}_{-m_1}(x)$ then $x\in B^{\bfX,k}_{-m_\DimaA{1}}(y)$.
        \item\label{lem:tri:ultra} If the rectification of $\bfX$ is simple and  $m_1\geq m_2$, then
        $$B_{-m_1}^{\bfX,k}(x)\sub B_{-m_2}^{\bfX,k}(x)=B^{\bfX,k}_{-m_2}(y)$$ for any $x\in \bfX(\F_{\ell^k}((t)))$ and $y\in  B_{-m_2}^{\bfX,k}(x)$.
    \end{enumerate}
\end{lem}

The following lemma says that regular maps are uniformly continuous and bounded on balls, in a way which is also uniform on the residue field:

\begin{lemma}[{\cite[Proposition 3.\RamiG{12}]{AGKS0}}]\label{lem:cont.bnd}
Let $\gamma:\bfX\to \bfY$  be a map of rectified algebraic varieties.  Then 
for any $m\in \N$ there is $m'>m$ s.t. for any $k$ and any $x\in B^{\bfX,k}_m$ we have 
\begin{enumerate}[(i)]
    \item \label{lem:cont.bnd:S1}$\gamma(B^{\bfX,k}_{m})\subset B^{\bfY,k}_{m'}$.
    \item \label{lem:cont.bnd:S2}$\gamma(B^{\bfX,k}_{-m'}(x))\subset B^{\bfY,k}_{-m}(\gamma(x))$.
\end{enumerate}
\end{lemma}
The following two corollaries imply that all the statements on balls that we formulate do not depend on the rectification.

\begin{cor}[{\cite[Corollary 3.\RamiG{13}]{AGKS0}}]\label{lem:indep-rect}    
    Let $\bfX_1, \bfX_2$ be two copies of the same $\F_\ell$-variety with two (possibly different) rectifications. Let $\bfZ\subset \bfX_1$  be a closed subvariety. Then 
    
\begin{enumerate}
\item\label{lem:indep-rect:1}    
    for any $m\in \N$ there is  $m' \in \N$ s.t. for any $k\in \N$ we have:
    \begin{enumerate}
        \item\label{lem:indep-rect:1a} $B_m^{\bfX_1,k}\subset B_{m'}^{\bfX_2,k}$.
        \item for any $x\in \RamiG{B_m^{\bfX_1,k}}$ we have $B_{\RamiG{-m'}}^{\bfX_1,k}(x)\subset B_{\RamiG{-m}}^{\bfX_2,k}(x)$
        \item\label{lem:indep-rect:1c} $B_{\RamiG{-m'}}^{\bfX_1,k}(\bfZ) \RamiG{\cap B_m^{\bfX_1,k}}\subset B_{\RamiG{-m}}^{\bfX_2,k}(\bfZ)$.    \end{enumerate}    
            \item \label{lem:indep-rect:2} 
            For any \RamiA{$\mu$-rectification}s of $\bfX_i$  and  $m\in \N$, there exists $m'$ s.t. for any $k$ we have:  $$\mu_m^{\bfX_1,k}< \ell^{km'}\mu_{m'}^{\bfX_2,k}$$
    \end{enumerate}    
\end{cor}

\begin{cor}[{\cite[Lemma 6.2]{AGKS0}}]
    The property of a map $\gamma:\bfX\to \bfY$ being effectively surjective does not depend on the rectifications on the varieties $\bfX$ and $\bfY$.
\end{cor}

At some point we will need the following stronger version of \Cref{lem:indep-rect}
\eqref{lem:indep-rect:1a}:
\begin{lem}[{\cite[Lemma 3.\RamiG{14}]{AGKS0}}]\label{lem:kot.bnd}
    Let $\bfX_1, \bfX_2$ be two copies of the same $\F_\ell$-variety with two (possibly different) rectifications. 
Then there exists $a\in \N$ s.t. 
    for any $m,k\in \N$ we have:
$$B_m^{\bfX_1,k}\subset B_{am+a}^{\bfX_2,k}.$$
\end{lem}

The next result is an effective version of the open mapping theorem:
\begin{lemma}[{\cite[Theorem 4.2]{AGKS0}}]\label{lem:eff.imp}
Let $\gamma:\bfX\to \bfY$ be a smooth map of smooth \Rami{(rectified)} algebraic varieties.  Then for any $m$ there is $m'$ s.t. for any $k$ and any $x\in B^{\bfX,k}_m$ we have $$\gamma(B^{\bfX,k}_{-m}(x))\supset B^{\bfY,k}_{-m'}(\gamma(x)).$$
\end{lemma}

The following is a criterion for effective subjectivity.
\begin{proposition}[{\cite[Theorem 6.3]{AGKS0}}]\label[proposition]{prop:crit-effsurj}
        Let $\gamma:\bfX\to \bfY$ be a smooth map of algebraic varieties that is  onto on the level of points for any   field. Then $\gamma$ is effectively surjective.
\end{proposition}

The following four statements  describe the behavior of the measures defined in \S\ref{ssec:notions} under push forward by a submersion. In particular they imply that all the statements that we formulate on these measures do not depend on the rectification.

\begin{lem}[{\cite[Lemma 3.\RamiG{16}]{AGKS0}}]\label{lem.sub.bnd}
    Let $\gamma: \bfX\to \bfY$ be  a submersion of \RamiA{$\mu$-rectified} varieties. Then for any $m$ there is $m'$ s.t. $$\gamma_*(\mu_m^{\bfX,k})< \ell^{km'}\mu_{m'}^{\bfY,k}.$$
\end{lem}

\begin{cor}[{\cite[Corollary 3.1\RamiG{7}]{AGKS0}}]\label[cor]{cor:bnd.int}
    Let $\bfX$ be a \RamiA{$\mu$-rectified} variety. Then for any $m\in \N$ there exists $M\in \N$ s.t. for any $k\in \N $:
    $$\mu_m^{\bfX,k}(B_\infty^{\bfX,k})< \ell^{kM}.$$
\end{cor}

\begin{lemma}[{\cite[Corollary 6.8]{AGKS0}}]\label{cor:push.bnd.above}    
    Let $\gamma: \bfX\to \bfY$ be  a submersion of \RamiA{$\mu$-rectified} varieties. Assume that $\gamma$ is effectively surjective. Then for any $m\in \N$ there is $m'\in \N$ s.t. for any $k\in \N$  we have $$\mu_{m}^{\bfY,k} < \ell^{km'}\gamma_*(\mu_{m'}^{\bfX,k})$$
\end{lemma}

\begin{lemma}[{\cite[Theorem 5.7]{AGKS0}}]\label{lem:push.mes.sm}
    Let $\gamma:\bfX_1 \to \bfX_2$ be a  smooth map  of  \RamiB{$\mu$-rectified} varieties. Then for any $m\in \N$ there \RamiG{are}  $\RamiG{m''>}m' \in \N$ such that for any $k\in \N$, \RamiG{$0 \leq m_0\leq m$,} and any $m$-smooth function $g\in C^\infty_\RamiG{c}(B_{\infty}^{\bfX_1,k})$ there is an $\RamiG{m''}$-smooth function $f\in C^\infty_c(B_{m'}^{\bfX_2,k})$ such that:
    $$\gamma_*(g\mu_{\RamiG{m_0}}^{\bfX_1,k})=f\cdot \mu_{m'}^{\bfX_2,k}.$$
\end{lemma}

\section{Basic geometry of the Chevalley map $p$}\label{sec:basic}
We recall that $p:\ug\to \uc$, is the Chevalley map sending \Rami{a matrix} to its characteristic polynomial. 
We need \Rami{further notation:}
\begin{notation}
$ $
\begin{enumerate}
    \item 
Let $\ut$\index{$\ut$} be the standard Cartan subalgebra of $\ug$, consisting of diagonal matrices.
    \item Let $W:=S_n$\index{$W$} be the Weyl group of $\GL_n$.
\item Identify $\fc\cong \ft//W$ by the Chevalley restriction theorem (see e.g. \cite[\S 23]{Humph}). We will also identify it with  
the affine space $\A^n$.
    \item  Denote by $q:\ut\to \uc$\index{$q$} the quotient map.  
    \item Denote by $\ug^{rss}$\index{$\ug^{rss}$} the locus of regular semi-simple elements (i.e. matrices with $n$ different eigenvalues over the algebraic closure).
    \item Denote $\uc^{rss}:=p(\ug^{rss})$ the locus of separable polynomials in $\uc$.\index{$\uc^{rss}$}
\end{enumerate}
\end{notation}

The following lemma follows immediately from miracle flatness (see \cite[\href{https://stacks.math.columbia.edu/tag/00R4}{Lemma 00R4}]{SP}):
\begin{lemma}\label{lem:flatFactor}
    $q:\ut\to \uc$ is flat.
\end{lemma}
\begin{notation}    
Let $I$ be a finite set.
For $w \in \mathbb{Z}^I$ define the $w$-degree\index{$w$-degree} of a monomial 
$\prod_{i\in I} x_i^{a_i}$ to be $\sum_{i\in I} a_i w(i)$.
For a polynomial 
$f \in k[\mathbb{A}^I]$,  we denote by $\sigma_w(f)$\index{$\sigma_w(f)$} the sum of the monomials of $f$ with the
highest $w$-degree. 
\end{notation}

\begin{proposition}    
\label[proposition]{prop:elimination}
Let $I, J$ be finite sets and let 
$\psi = (\psi_j)_{j \in J} : \mathbb{A}^I \to \mathbb{A}^J$
be a morphism such that $\psi(0) = 0$, fix $w \in \mathbb{Z}^I$.
Let $\sigma_w(\psi) = (\sigma_w(\psi_j))_{j \in J}$. If $\sigma_w(\psi)$ is flat at $0$, then so is $\psi$.
\end{proposition}
\begin{proof}
    The proof is identical to the proof of \cite[Proposition 2.1.16]{AA}, replacing \cite[Proposition 2.1.1]{AA} with \cite[IV, 11.3.10]{EGA}.
\end{proof}
\begin{cor}\label[cor]{cor:pqflat}
    $p:\ug\to \uc$ is flat.
\end{cor}
\begin{proof}
Let $I=\{1,\dots,n\}\times \{1,\dots,n\}$ be the set of indices of $n\times n$ matrices. Following \cite{BL} we let $w:I\to \Z$ defined by $w(i,j)=\delta_{i,j}$. Identify $\ug$ with $\bA^I$. Using 
Proposition 
\ref{prop:elimination}  and \Cref{lem:flatFactor}
 we obtain that $p$  is flat at $0$. The assertion follows, using the homothety action  of $\bG_m$ on $\ug$.
\end{proof}

\begin{lem}\label{lem:fib.p.irred}
The fibers of $p$ are irreducible. 
\end{lem}
\begin{proof}
    This follows from the Jordan decomposition.    
\end{proof}

\begin{notation}
Denote by $\ug^r$\index{$\ug^r$} the smooth locus of $p$.
\end{notation}

\begin{lem}\label{lem:Gr.onto}
    $p|_{\ug^r}:\ug^r\to \uc$ is onto.
\end{lem}
\begin{proof}
    The companion matrix $C(f)$ attached to a polynomial $f \in \uc $ is a regular matrix with characteristic polynomial equal to $f$. This proves the assertion.
\end{proof}

\begin{cor}\label[cor]{cor:pFibReduced}
The fibers of $p$ are absolutely 
reduced. Furthermore, $\ug^r$ is big in $\ug$.
\end{cor}

\begin{proof}
By \Cref{lem:Gr.onto}, each fiber of $p$ has a generically reduced component. Hence by \Cref{lem:fib.p.irred}, the fibers of $p$ are  generically reduced. Since $p$ is flat, its fibers are complete intersections. Thus by \cite[Exercise 18.9]{Eis}, they are reduced. 


To show that $\ug^r$ is big in $\ug$ we let $\bfZ$ be its complement inside $\ug$.  By the above, for any $c\in\uc$ we have $\dim(p^{-1}(c) \cap \bfZ)<\dim(\bfZ)$.
We obtain:
\begin{align*}
\dim(\bfZ) &\leq \dim(\uc-\uc^{rss})+\max_{c \in \uc-\uc^{rss}}\dim(p^{-1}(c) \cap \bfZ)\leq
\\&\leq
\dim(\uc)-1+\max_{c \in \uc-\uc^{rss}}(\dim p^{-1}(c))-1=\dim(\ug)-2
\end{align*}
\end{proof}

    Recall that for an integer $i\in \N$ we denote by
    $\ug_i:=\ug^{\times_\uc i}$ the $i$-folded fiber product of $\ug$ with itself over $\uc$ with respect to the map $p$. A-priory this is a scheme, but we can now show that it is reduced, so it is a variety.

\begin{lem}\label{lem:X.i.red}
The scheme $\gi$ is reduced.   
\end{lem}
\begin{proof}
By \Cref{cor:pFibReduced} 
the fibers of $p$ are reduced. Therefore, so are the fibers of  $\gi\to \uc$. This implies the assertion.
\end{proof}

\section{Effectively an-FRS over the origin}\label{sec:EfanFRS0}
In this section we introduce a class of maps of algebraic variety to a vector space that refines the notion of an-FRS maps (See \Cref{def:an.FRS.0}). 

The definition of this property is designed in a way that adapts the original Harish-Chandra's argument for an-FRS property of the Chevalley map to give some result also in positive characteristic. 
This enables us to prove this property for the Chevalley map (see \Cref{thm:an.jet} below).
We later use this property in order to prove \Cref{thm:jet.nil}.

\begin{definition}
    We say that an action of $\bG_m$ on an affine space is positive\index{positive action} if, in an appropriate coordinate system, it is given by $$\lambda \cdot (x_1,\dots,x_M)=(\lambda^{a_1}x_1,\dots,\lambda^{a_M}x_M),$$ with all $a_i$ positive.
\end{definition}

\begin{definition}[effectively an-FRS over the origin]\label[definition]{def:an.FRS.0}
    Let $\bfX$ be a rectified variety and let $\gamma:\bfX \to \bfY$ be an algebraic map to an affine space $\bfY$. 
    Let $\bG_m$ act positively on   $\bfY$.
    Choose the standard rectification of $\bfY$.
    
We say that $\gamma$ is effectively an-FRS over the origin\index{effectively an-FRS over the origin} if
for any $m\in \N$ there exist $M$ such that
for any $k\in \N$ and $a\in \Z$ we have
$$\frac{\left\langle \gamma_{*}\left(\mu_m^{\bfX,k} \right),1_{t^{a}\cdot B_0^{\bfY,k}} \right\rangle}{\mu^{\bfY,k}\left (t^{a}\cdot B_0^{\bfY,k}\right)}< \ell^{kM}$$
\end{definition}
The following follows immediately from \Cref{lem:indep-rect}\eqref{lem:indep-rect:2}.
\begin{lemma}
The notion of effectively an-FRS over the origin does not depend on the rectification on the source. 
\end{lemma}

In the this subsection we will prove the following:
\begin{theorem}\label[theorem]{thm:an.jet}
    The Chevalley map $p$ is effectively an-FRS over the origin.
\end{theorem}
Let us briefly explain the idea of the proof of this theorem:
\begin{itemize}
    \item We first prove 2 statements on the \Rami{effective} an-FRS over the origin property:
    \begin{itemize}
        \item It is local in the smooth topology (on $\bfX$). See \Cref{cor:sm.loc}
        \item For $\bG_m$-equivariant maps between affine spaces, one can give a criterion for the \Rami{effective} an-FRS over the origin property in terms of the exponents of the actions of $\bG_m$. See \Cref{lem:FRS.hom.crit} below.
        \end{itemize}
        \item We then use the original Harish-Chandra's argument:
        \begin{itemize}
            \item Proof by a descending induction that the Chevalley map restricted to the complement of an invariant closed subset of $\bfN$ is \Rami{effectively} an-FRS over the origin.
            \item The base of the induction follows from the fact that in this case \RamiC{the origin} is not in the image of the  map. The fact that such maps are \Rami{effectively} an-FRS over the origin is not completely obvious, but we prove it in \Cref{lem:not.in.im} below.
            \item For the step of the induction we use \Rami{an}  analog of the Slodowy slice (see \Cref{lem:slod.slice} below). We use the locality of the \Rami{effective} an-FRS over the origin property in order to deduce the statement from an analogous statement for the slice.
            \item We prove the analogous statement for the slice using the criterion for \Rami{the effective} an-FRS over the origin property for $\bG_m$-equivariant maps.
        \end{itemize}    
\end{itemize}
In the next 2 subsections we provide some preparations for actual proof which will be given in \S\ref{sec:Pf.an.jet}. 
\subsection{Basic properties of effectively an-FRS maps over \RamiC{the origin}}

\begin{lemma}\label{lem:FRS.hom.crit}
    Let $\bfX=\A^I$ and $\bfY=\A^J$ be affine spaces with positive actions of $\bG_m$ given by $$s \cdot (x_1,\dots,x_I)=(s^{a_1}x_1,\dots,s^{a_I}x_I)$$ and $$s \cdot (y_1,\dots,y_J)=(s^{b_1}y_1,\dots,s^{b_J}y_J)$$ respectively. 
    Assume that $\sum_{i=1}^I  a_i > \sum_{j=1}^J  b_j$.
    
    Let $\varphi:\bfX\to \bfY$ be an equivariant map such that $\varphi|_{\bfX\smallsetminus 0}:(\bfX\smallsetminus 0)\to \bfY$ is effectively an-FRS over the origin. Then 
     $\varphi:\bfX\to \bfY$  is effectively an-FRS over the origin.
\end{lemma}
\begin{proof}

We begin by comparing the balls in ${\bfX\smallsetminus 0}$ to spheres in ${\bf X}$. More precisely, we fix $m\in \bN$. For any
$k,i\in\N$, we denote by  $$S^{\bfX,k}_{m,i}:= t^{i}\cdot B^{\bfX,k}_m \smallsetminus t^{i+1}\cdot B^{\bfX,k}_m$$ a sphere in ${\bf X}$.
To make the comparison, we choose standard rectifications on $\bfX$ and choose the rectification on $\bfX\smallsetminus 0$ given by the cover with the complements of coordinate hyperplanes and the induced forms from the standard form on $\bfX$. It is easy to see that there exists $m'>m$ s.t. for any $k$ we have   $$B^{\bfX\smallsetminus 0,k}_{m'} \supset S^{\bfX,k}_{m,0}$$
    and moreover $$ \ell^{km'} \mu_{m'}^{\bfX\smallsetminus 0,k}>1_{S^{\bfX,k}_{m,0}} \mu_m^{\bfX,k}.$$
    
    We choose the standard rectification on ${\bf Y}$ and use the above comparison to deduce that 
    for any $k\in \N$ and $a\in \Z$ we have:
$$
\frac{\left\langle \varphi_{*}\left(1_{S_{m,0}^{\bfX,k}} \mu_m^{\bfX,k} \right),1_{t^{a}\cdot B_0^{\bfY,k}} \right\rangle}{\mu^{\bfY,k}\left (t^{a}\cdot B_0^{\bfY,k}\right)}< 
\ell^{km'}
\frac{\left\langle \varphi_{*}\left(\mu_{m'}^{\bfX\smallsetminus 0,k} \right),1_{t^{a}\cdot B_0^{\bfY,k}} \right\rangle}{\mu^{\bfY,k}\left (t^{a}\cdot B_0^{\bfY,k}\right)}$$
    On the other hand, 
    the assumption that $\varphi|_{\bfX\smallsetminus 0}:(\bfX\smallsetminus 0)\to \bfY$ is effectively an-FRS over the origin implies that there exists $N_0$ s.t. 
    for any $k\in \N$ and $a\in \Z$ we have
$$
\ell^{km'}
\frac{\left\langle \varphi_{*}\left(\mu_{m'}^{\bfX\smallsetminus 0,k} \right),1_{t^{a}\cdot B_0^{\bfY,k}} \right\rangle}{\mu^{\bfY,k}\left (t^{a}\cdot B_0^{\bfY,k}\right)}< 
\ell^{kN_0}$$
Combining the last two inequalities we obtain that for any $k\in \N$ and $a\in \Z$ we have:
\begin{equation}\label{=comb}
\frac{\left\langle \varphi_{*}\left(1_{S_{m,0}^{\bfX,k}} \mu_m^{\bfX,k} \right),1_{t^{a}\cdot B_0^{\bfY,k}} \right\rangle}{\mu^{\bfY,k}\left (t^{a}\cdot B_0^{\bfY,k}\right)}<\ell^{kN_0}    
\end{equation}

By the equivariance of $\varphi$ we obtain that
for any integers $i,k\in\N$ we have:
$$
   \varphi_{*}\left(1_{S_{m,i}^{\bfX,k}} \mu_m^{\bfX,k} \right) 
      =
     t^{i}\cdot \left(\varphi_{*}\left(t^{-i}\cdot \left(1_{S_{m,i}^{\bfX,k}} \mu^{\bfX,k} \right)\right)\right) 
$$
We also have $$     t^{-i}\cdot \left(1_{S_{m,i}^{\bfX,k}} \mu^{\bfX,k} \right) =  \ell^{-ki\sum a_j}1_{S_{m,0}^{\bfX,k}} \mu^{\bfX,k}$$
and similarly
$$
t^i\cdot 1_{t^{a}\cdot B_0^{\bfY,k}}=1_{t^{a-i}\cdot B_0^{\bfY,k}} \text{ and }
\mu^{\bfY,k}\left (t^{a}\cdot B_0^{\bfY,k}\right)=
\ell^{-ki\sum b_j}
\mu^{\bfY,k}\left (t^{a-i}\cdot B_0^{\bfY,k}\right)
$$
Thus, for any integers $i\in\N$, $a\in  \Z$, and $k\in \N$, we obtain:
\begin{align*}
    \frac{\left\langle \varphi_{*}\left(1_{S_{m,i}^{\bfX,k}} \mu_m^{\bfX,k} \right),1_{t^{a}\cdot B_0^{\bfY,k}} \right\rangle}{\mu^{\bfY,k}\left (t^{a}\cdot B_0^{\bfY,k}\right)}
    &=
    \ell^{-ki\sum a_j}\frac{\left\langle \varphi_{*}\left(\left(1_{S_{m,0}^{\bfX,k}} \mu^{\bfX,k} \right)\right),1_{t^{a-i}\cdot B_0^{\bfY,k}} \right\rangle}{\mu^{\bfY,k}\left (t^{a}\cdot B_0^{\bfY,k}\right)}
    \\&=
    \frac{\ell^{-ki\sum a_j}}
    {\ell^{-ki\sum b_j}}
    \cdot 
    \frac{\left\langle \varphi_{*}\left(\left(1_{S_{m,0}^{\bfX,k}} \mu^{\bfX,k} \right)\right),1_{t^{a-i}\cdot B_0^{\bfY,k}} \right\rangle}{\mu^{\bfY,k}\left (t^{a-i}\cdot B_0^{\bfY,k}\right)}
\end{align*}
Combining this with the assumption $\sum a_j> \sum b_j$ and the inequality \eqref{=comb} we obtain 
$$
    \frac{\left\langle \varphi_{*}\left(1_{S_{m,i}^{\bfX,k}} \mu_m^{\bfX,k} \right),1_{t^{a}\cdot B_0^{\bfY,k}} \right\rangle}{\mu^{\bfY,k}\left (t^{a}\cdot B_0^{\bfY,k}\right)}
  <
    \ell^{-ki+kN_0}$$

Using the fact that
$
B^{\bfX,k}_m =\bigsqcup_{i=0}^{\infty}S^{\bfX,k}_{m,i}
$ 
we deduce that for any integers $a\in  \Z$, and $k\in \N$, we have:
\begin{align*}
    \frac{\left\langle \varphi_{*}\left(\mu_m^{\bfX,k} \right),1_{t^{a}\cdot B_0^{\bfY,k}} \right\rangle}{\mu^{\bfY,k}\left (t^{a}\cdot B_0^{\bfY,k}\right)}
    &=\sum_{i=0}^\infty
    \frac{\left\langle \varphi_{*}\left(1_{S_{m,i}^{\bfX,k}} \mu_m^{\bfX,k} \right),1_{t^{a}\cdot B_0^{\bfY,k}} \right\rangle}{\mu^{\bfY,k}\left (t^{a}\cdot B_0^{\bfY,k}\right)}
        \\&<\sum_{i=0}^\infty
    \ell^{-ki+kN_0}\leq \ell^{k(N_0+1)}.
\end{align*}
Taking $M:=N_0+1$ we get the required bound.
\end{proof}
\Cref{lem.sub.bnd},  \Cref{cor:push.bnd.above}, and \Cref{prop:crit-effsurj} give us:

\begin{cor}\label[cor]{cor:sm.loc}
    Let $\bfX$ be a \RamiA{$\mu$-rectified} variety and let $\gamma:\bfX \to \bfY$ be an algebraic map to an affine space $\bfY$. 
    Let $\bG_m$ act positively on $\bfY$.
    
    Let $\delta:\tilde \bfX \to \bfX$ be a submersion.
    
    Then:
    \begin{enumerate}
        \item if $\gamma$ is effectively an-FRS over the origin  then so is $\gamma \circ \delta$.
        \item if $\delta$ effectively surjective and  $\gamma \circ \delta$ is effectively an-FRS over the origin  then so is $\gamma$.
        \item if $\delta$ is onto on the level of \Rami{points} for any field and  $\gamma \circ \delta$ is effectively an-FRS over the origin  then  $\gamma$ is also effectively an-FRS over the origin. 
    \end{enumerate}
\end{cor}
\begin{proof}$ $
    \begin{enumerate}
        \item Follows from \Cref{lem.sub.bnd}.
        \item Follows from \Cref{cor:push.bnd.above}.
        \item Follows from (2) and \Cref{prop:crit-effsurj}.
    \end{enumerate}
\end{proof}
\begin{lem}\label{lem:not.in.im}
    Let $\bfX$ be a \RamiA{$\mu$-rectified} variety and let $\gamma:\bfX \to \bfY$ be an algebraic map to an affine space $\bfY$. 
    Let $\bG_m$ act positively on $\bfY$.
    Assume that $0\notin \gamma(\bfX(\bar\F_\ell))$. Then $\gamma $ is effectively an-FRS over the origin.
\end{lem}
\begin{proof}
Fix $m\in \N$.
By \Cref{cor:bnd.int}, there exists $M_1\in \N$ s.t. for any $k\in \N $:
    $$\mu_m^{\bfX,k}(B_\infty^{\bfX,k})< \ell^{kM_1}.$$    
Let $\lambda_1,\dots, \lambda_I$ be the exponents of the $\bG_m$ action on $\bfY$.
Choose the rectification on $\bfY':=\bfY\smallsetminus \{0\}$ given by the open cover of $\bfY'$  by compliments to coordinate hyperplanes, and their standard embeddings into $\bfY'\times \A^1$.
Let $\gamma':\bfX\to \bfY\smallsetminus \{0\}$ be s.t. $\gamma$ factors through $\gamma'$. By \Cref{lem:cont.bnd} there exists $m_1$ s.t. for any $k\in \N $  we have $\gamma'(B_m^{\bfX,k})\subset  B_{m_1}^{\bfY\smallsetminus \{0\},k}$. Take $$M=M_1+k(m_1+1)\sum_i\lambda_i+1.$$
Note that  $$B_{m_1}^{\bfY\smallsetminus 
\{0\},k}= B_{m_1}^{\bfY,k}\smallsetminus B_{-m_1-1}^{\bfY,k}(0)$$

Fix $a\in \Z$ and $k\in \N$.

\begin{enumerate}[{Case} 1.]
    \item $a>\frac{m_1+1}{\min_i \lambda_i}$:
$$\frac{\left\langle \gamma_{*}\left(\mu_m^{\bfX,k} \right),1_{t^{a}\cdot B_0^{\bfY,k}} \right\rangle}{\mu^{\bfY,k}\left (t^{a}\cdot B_0^{\bfY,k}\right)}\leq \frac{\left\langle \gamma_{*}\left(\mu_m^{\bfX,k} \right),1_{ B_{-m_1-1}^{\bfY,k}}(0) \right\rangle}{\mu^{\bfY,k}\left (t^{a}\cdot B_0^{\bfY,k}\right)}=0< \ell^{kM}$$
    \item $a\leq\frac{m_1+1}{\min_i \lambda_i}$:
$$\frac{\left\langle \gamma_{*}\left(\mu_m^{\bfX,k} \right),1_{t^{a}\cdot B_0^{\bfY,k}} \right\rangle}{\mu^{\bfY,k}\left (t^{a}\cdot B_0^{\bfY,k}\right)}\leq
\frac{\mu_m^{\bfX,k}(\bfX(\F_{\ell^k}((t))))}{\mu^{\bfY,k}\left (t^{m_1+1}\cdot B_0^{\bfY,k}\right)}\leq 
\frac{\ell^{kM_1}}{\ell^{-k(m_1+1)\sum_i\lambda_i}}
< \ell^{kM}$$
\end{enumerate}
\end{proof}

\subsection{Slices to nilpotent elements}\label{sec:slices}
We will need existence of slices for nilpotent orbits in $\fg$ that have a $\bG_m$ action satisfying an appropriate condition. This is analogous to the approach of Harish-Chandra's in his proof  of the boundedness of normalized orbital integrals. There the slice was defined by an $\fs\fl_2$ triple (a.k.a. Slodowy slice). However, we can not rely on $\fs\fl_2$ triples, so we have to do it in a more ad-hoc procedure.
\begin{lemma}
\label{lem:slod.slice} 
    Let $x\in \ug(\F_\ell)$ be a nilpotent element. Then there are:
    \begin{itemize}
        \item a linear subspace $\bfL\subset \ug$
        \item a positive $\bG_m$ action on $\bfL$.    
    \end{itemize}
     s.t.
     \begin{enumerate}
        \item\label{lem:slod.slice:2} the action map $\bfG\times (x+\bfL)\to \ug$ is a submersion.        
        \item\label{lem:slod.slice:3}  For any nilpotent $y \in x+(\bfL(\bar\F_\ell)\smallsetminus 0)$  we have $\dim \bfG \cdot y > \dim \bfG \cdot x$ 
       \item \label{lem:slod.slice:4} The Chevalley map $p|_{x+\bfL}$ intertwines the $\bG_m$ action on $x+\bfL$ (given by the identification $y\mapsto x+y$ between $\bfL$ and $x+\bfL$) with the $\bG_m$ action on $\uc$ given by $(\lambda \cdot f)(z):=\lambda^{n}f(\lambda^{-1} \cdot z)$.
        \item\label{lem:slod.slice:5}  If $x$ is not regular then the sum of the exponents of the $\bG_m$ action on $\bfL$ is larger than $\frac{n(n+1)}{2}$.
     \end{enumerate}    
\end{lemma}    

\begin{remark}
Similar results were proven in various contexts, see e.g. 
     \cite{Pre}. Since we would like to have an explicit construction, we include the proof here for completeness.
\end{remark}
For the proof we will need some notations:
\begin{notation}\label{not:slice}
$ $
\begin{enumerate}

\item \label{action of GxG_m}
Consider the action of $\bfG(F)\times F^\times$ on $\ug(F)$ where the first coordinate acts by conjugation and the second by homothety.   
We denote this action by $(g,\lambda) \cdot A=\lambda^{-1}gAg^{-1}$.
    \item For an integer $k>0$ let $J_{k}$\index{$J_{k}$} be the $k\times k$ nilpotent Jordan block.
    \item Given a matrix of of matrices $\{A_{i,j}\}_{i,j\in \{1,\dots,j\}}$ we denote by $Bl(\{A_{i,j}\})$\index{$Bl$} the corresponding block matrix. Here we assume that the sizes of $A_{i,j}$ matches.
    \item For 2 integers $n_1,n_2$ set $$\bfV_{n_1,n_2}':=
\{\{a_{kl}\}\in Mat_{n_1,n_2}|a_{kl}=0 \text{ for } l>1   \}.$$\index{$\bfV_{n_1,n_2}'
    and  $$\bfV_{n_1,n_2}:=
\{\{a_{kl}\}\in \bfV_{n_1,n_2}'|  \text{ we have } a_{k1}=0; \text{ for } k\leq n_1-n_2 \}$$\index{$\bfV_{n_1,n_2}$}$}
\item For an integer $k$ we set 
$t_{k}:\bG_m\to \GL_{k}$
defined by 
$$t_{k}(\lambda)=\diag(1,\lambda,\dots,\lambda^{k-1}).$$\index{$t_{k}$}
\item For a nilpotent element in Jordan canonical form $x\in \ug(\F_\ell)$ with blocks of sizes $n_1,\dots, n_k$ we
define  maps $t_{i,j}:\bG_m\to {\bf \mathrm{Mat}}_{n_i\times n_j}$ 
by  
$$
t_{i,j}(\lambda):=
\begin{cases}
t_{n_i}(\lambda) & \text{ for } i= j\\
0 & \text{ otherwise }
\end{cases}\index{$t_{i,j}$}
$$
Define $t_x:\bG_m\to \bfG$ by 
$$t_x(\lambda)=Bl(\{t_{i,j}(\lambda)\}).$$
\item \label{co-character.for. Jordan nilpotent} For $x$ as above
define $\phi_x:\bG_m\to \bfG\times \bG_m$ by 
$$\phi_x(\lambda)=(t_x(\lambda),\lambda^{-1}).$$\index{$\phi_x$}
\item \label{slice.for. Jordan nilpotent} For $x$ as above, define 
$$\bfL_x=\{Bl(\{A_{i,j}\}_{i,j})| A_{i,j}\in \bfV_{n_i,n_j}\}\subset \ug.$$\index{$\bfL_x$}  
\end{enumerate}
\end{notation}
\begin{proof}[Proof of \Cref{lem:slod.slice}]
We now construct our slices. WLOG we can assume that $x$ is a Jordan canonical form.
We define:
\begin{enumerate}[(i)]
    \item 
A linear space $\bfL:=\bfL_x$, see Notation \ref{not:slice}(\ref{slice.for. Jordan nilpotent}). 
\item 
An action of $\bG_m$ on $\bfL$ by $$\lambda\star A:=\phi_x(\lambda) \cdot A,$$ see Notation \ref{not:slice}(\ref{action of GxG_m},\ref{co-character.for. Jordan nilpotent}).
\end{enumerate}
It is easy to see that 
 this is a positive action. It is evident that \Rami{condition \eqref{lem:slod.slice:4}} is satisfied: 
 $$p(\lambda \star A)=p(\phi_x(\lambda) \cdot A))=p(\lambda Ad(t_x(\lambda))(A))=p(Ad(t_x(\lambda))(\lambda A))=p(\lambda A)=\lambda \cdot p(A).$$
 We now verify \Rami{condition \eqref{lem:slod.slice:5}}. 
Assume $x$ is not regular, and let $n_1 \geq n_2 \cdots \geq n_k$ be the corresponding partition of $n,$ with $k\geq 2$. 

It is a simple verification that:
\begin{enumerate}[(a)]
    \item 
In the block $n_i \times n_i$ the exponents are
$1,2,\cdots, n_i$.
\item For $i<j$ we have two blocks $n_i \times n_j$ where the exponents are $$n_i-n_j+1,\cdots,n_i.$$
\item For $i>j$ we have two blocks $n_i \times n_j$ where the exponents are $$1,\cdots, n_j.$$
\end{enumerate}
 To sum up we obtain that the sum of the ${\bG}_m$ exponents is:
 \begin{multline*}
     \sum_{i=1}^k \frac{1}{2} n_i(n_i+1) + \sum_{1 \leq i<j \leq k} (n_in_j+n_j)=\frac{1}{2}(\sum_{i=1}^k  n_i(n_i+1) +\sum_{1 \leq i<j \leq k} (2n_in_j+2n_j)) \\ 
 =\frac{1}{2}(n^2+n)+\sum_{j=1}^k(j-1)n_j > \frac{1}{2}(n^2+n).
 \end{multline*}
 
 It remains to prove \Rami{conditions (\ref{lem:slod.slice:2},\ref{lem:slod.slice:3})}.
 For \Rami{\eqref{lem:slod.slice:2}} we first note that the action map is submersive at point $(1,x)$. Then we deduce \Rami{\eqref{lem:slod.slice:2}} from the $\bG_m$ action $\star$. 
  
 For \Rami{\eqref{lem:slod.slice:3}} take a nilpotent element $y \in x+(\bfL(\bar\F_\ell)\smallsetminus 0)$. Using the action $\star$ we see that $\overline{(\bfG\times \bG_m)\cdot y}\ni x$. Since $y$ is nilpotent, we have $(\bfG\times \bG_m)\cdot y=\bfG\cdot y$. Thus 
 $\overline{\bfG\cdot y}\ni x$. Assume, for the contradiction, that $\dim(\bfG\cdot y)\leq\dim(\bfG\cdot x)$. We obtain $\bfG\cdot y\ni x$.
It remains to show that $\bfG\cdot x\cap x+\bfL_x=\{x\}$. This follows from the following easily verified facts:
\begin{itemize}
    \item $x$ is an isolated point of the intersection $\bfG\cdot x\cap x+\bfL_x$. 
    \item the intersection $\bfG\cdot x\cap x+\bfL_x$ is $\bG_m$-invariant w.r.t. the action $\star$ and 
    the closure of any such $G_m$-orbit of a point in the slice $x+{\bf L}_x$ includes $x$.
\end{itemize}
To verify the first point we claim that the intersection of the tangent space $T_{x}(G \cdot x)$ and $L_x$ is zero. 
For $x=J_n$ this follows from the following statement:
If the matrix $[J_n,A]$ that has the first column equals to zero then $[J_n,A]=0$. 
A similar argument works for any direct sum of Jordan blocks.
\end{proof}

\subsection{Proof of \Cref{thm:an.jet}}\label{sec:Pf.an.jet}
\begin{proof}
Let $\bfN\subset \ug$ be the nilpotent cone. Enumerate  the nilpotent orbits $\{0\}=\bfO_1,\dots \bfO_m$ s.t. $\dim \bfO_i\leq \dim \bfO_j$ for any $i<j$. Let $\bfN_i=\bigcup_{j=1}^i \bfO_j$. Note that $\bfN_i$ is closed and $\bfN_0=\emptyset$.
We will prove  by down going induction on $i$ that for any $i\geq 0$ the map
$p|_{\ug \smallsetminus \bfN_i}$ is effectively an-FRS over the origin.

The base of the induction $i=m$ follows from \Cref{lem:not.in.im}.

For the induction step, 
we assume the statement holds for $\bfN_{i+1}$ and prove it for $\bfN_{i}$. Let $\bfU=\ug\smallsetminus \bfN_{i+1}$.
Let $x\in \bfO_{i+1}(\F_\ell)$ and let $\bfL$ be the linear space given by \Cref{lem:slod.slice} when applied to 
$x$.
\begin{enumerate}[Step 1.]
    \item Reduction to $\bfG\times \bfL$.\\
Consider the map $$\delta:(\bfG\times \bfL) \sqcup \bfU\to 
\ug\smallsetminus \bfN_i$$ given on $\bfG\times \bfL$ by $\delta(g,l):=g\cdot (x+l)$ and on $\bfU$ by the embedding $\bfU\subset \ug\smallsetminus \bfN_i$. By \Cref{lem:slod.slice} (\ref{lem:slod.slice:2}) $\delta$ is submersive. Also, it is onto on the level of points over any field. Indeed, for any extension $E/\F_\ell$  we have 
$\bfO_{i+1}(E)=\bfG(E)\cdot x$ and thus,
$$(\ug\smallsetminus \bfN_i)(E)=\bfO_{i+1}(E) \cup \bfU(E)= (\bfG(E)\cdot x )\cup \bfU(E)\subset \delta(((\bfG\times \bfL) \sqcup \bfU)(E))$$

Thus by \Cref{cor:sm.loc} it is enough to show that $p\circ \delta$ is effectively an-FRS over the origin. Let $\delta':=\delta|_{\bfG\times \bfL}$.
Notice that $p|_{\bf U}$ is effectively an-FRS over the origin by the induction hypothesis.

Therefore it is enough to show that $p\circ \delta'$ is effectively an-FRS over the origin.

Notice that 
$\delta' (\bfG\times (\bfL\smallsetminus 0)) \subset \bfU $ by \Cref{lem:slod.slice} \eqref{lem:slod.slice:3}. So,  by  \Cref{cor:sm.loc} we deduce 
that 
$p\circ \delta'|_{\bfG\times (\bfL\smallsetminus 0) }$ is effectively an-FRS over the origin. 
\item Reduction to $\bfL$.\\
We can factor the map $p\circ \delta'$ as $p|_{(x+\bfL)} \circ sh_x \circ pr_{\bfL}$, where $sh_x:\bfL\to x+\bfL$ is the shift map,  and $pr_{\bfL}:\bfG\times \bfL\to \bfL$ is the projection. So, by \Cref{cor:sm.loc} it is enough to show that $p|_{x+\bfL}\circ sh_x:\bfL \to \fc$ is effectively an-FRS over the origin.

Also, by \Cref{cor:sm.loc} we deduce 
that 
$p|_{x+(\bfL\smallsetminus 0)} \circ sh_x$ is effectively an-FRS over the origin.

\item Proof that $p|_{x+\bfL}\circ sh_x:\bfL \to \fc$ is effectively an-FRS over the origin.\\
If $x$ is regular nilpotent then $p|_{x+\bfL} \circ sh_x$ is an isomorphism, and thus is \Rami{effectively} an-FRS over the origin.

Otherwise, the assertion follows now from \Cref{lem:FRS.hom.crit}, since the condition of \Cref{lem:FRS.hom.crit} on the source $\bfL$ is given by \Cref{lem:slod.slice}\eqref{lem:slod.slice:5}.
\end{enumerate}

\end{proof}

\section{An-FRS maps and jets}\label{sec:jet}
In this section we relate the effective an-FRS property to jet schemes.
Specifically we prove:
\begin{proposition}    \label[proposition]{prop:eff.FRS.imp.jet}
    Let $\gamma:\bfX\to \bfY$ be a flat $\bG_m$-equivariant map between affine spaces  with positive $\bG_m$ actions. Assume that it is  effectively an-FRS over the origin. Then there exists $C\in \N$ s.t. 
    for every $m\in \bN$ we have
    $$\dim \mathcal{J}_m(\gamma^{-1}(0))<C+m\dim\gamma^{-1}(0).$$
\end{proposition}
\begin{proof}[Idea of the proof]
 The fact that $\gamma$ is effectively an-FRS over the origin provides a bound on the ratio between the measure of the preimage of a certain ellipsoid in $\bf Y$ and the measure of the ellipsoid itself. If this ellipsoid would be a  ball then this ratio would be exactly the (normalized) number of points of the jet scheme in question. So, we would get the desired bound on the dimension from the Lang-Weil bounds. In our case, we get a bound on the dimension of some other scheme ($\delta^{-1}(0)$ in the notation below). We can embed our ellipsoid into a ball. This ball is the union of shifts of the ellipsoid. So we need to bound the measures of the preimages of these shifts. The Lang-Weil bounds translate the problem  to a question on dimensions of neighboring fibers of the same map $\delta$. We can bound these dimensions using the semi-continuity of the dimension of the fiber.
\end{proof}
\begin{proof}
Choose the standard rectifications on $\bfX$ and $\bfY$.
Since $\gamma$ is  effectively an-FRS over the origin, there exists $M\in\N$ 
such that
for any $k,m\in \N$ we have
$$\frac{\left\langle \gamma_{*}\left(\mu_0^{\bfX,k} \right),1_{t^{m}\cdot B_0^{\bfY,k}} \right\rangle}{\mu^{\bfY,k}\left (t^{m}\cdot B_0^{\bfY,k}\right)}< \ell^{kM}.$$

Let $n:=\dim \bfY$ and let $\{a_i\}_{i=1}^n$ be the exponents of the $\bG_m$ action on $\bfY$. Set $a_{\max}=\max a_i$ and $a_{\min}=\min a_i$. Take $C:=M+a_{\min} \dim(\bfX)$.

    For $m\in\N$ consider the jet map $\mathcal{J}_m(\gamma):\mathcal{J}_m(\bfX)\to \mathcal{J}_m(\bfY)$. Let 
 $$E_m^k:=\left\{(\sum_j y_{1j} t^j,\dots, \sum_j y_{nj} t^j)\in \mathcal{J}_{a_{\max} m}(\bfY)(\F_{\ell^k})\,\,|\,\, y_{ij}=0 \text{ for } j<m a_i\right\}.$$

We observe that the measure of an ellipsoid in ${\bf Y}$ can be calculated by counting points of this set. Namely, 
$$
\mu^{\bfY,k}\left (t^{m}\cdot B_0^{\bfY,k}\right)=
\frac{\#E_{m}^{k}}{\ell^{kma_{\max}\dim(\bfY)}}.
$$

Also observe that
$$
\left\langle \gamma_{*}\left( \mu_0^{\bfX,k} \right),1_{t^{m}\cdot B_0^{\bfY,k}} \right\rangle =
\mu_0^{\bfX,k}(\gamma^{-1}(t^{m}\cdot B_0^{\bfY,k})) = \frac{\#\mathcal{J}_{ma_{\max}}(\gamma)^{-1}(E_m^k)}{\ell^{kma_{\max}\dim(\bfX)}}.
$$
We get 
\begin{align*}    
\#\mathcal{J}_{ma_{\max}}(\gamma)^{-1}(E_m^k)&= \ell^{kma_{\max}\dim(\bfX)}\left\langle \gamma_{*}\left( \mu_0^{\bfX,k} \right),1_{t^{m}\cdot B_0^{\bfY,k}} \right\rangle 
\\&=
\ell^{kma_{\max}\dim(\bfX)}\frac{\left\langle \gamma_{*}\left(\mu_0^{\bfX,k} \right),1_{t^{m}\cdot B_0^{\bfY,k}} \right\rangle}{\mu^{\bfY,k}\left (t^{m}\cdot B_0^{\bfY,k}\right)} {\mu^{\bfY,k}\left (t^{m}\cdot B_0^{\bfY,k}\right)}
\\&=
\ell^{kma_{\max}(\dim(\bfX)-\dim(\bfY))}\frac{\left\langle \gamma_{*}\left(\mu_0^{\bfX,k} \right),1_{t^{m}\cdot B_0^{\bfY,k}} \right\rangle}{\mu^{\bfY,k}\left (t^{m}\cdot B_0^{\bfY,k}\right)}  {\# E_m^k} 
\\&<
\ell^{kma_{\max}\dim(\gamma^{-1}(0))} \ell^{kM}  {\# E_m^k}, 
\end{align*}
where in the last step we used the fact that $\gamma$ is flat, as well as the inequality established above.

Let $\bfE_m\subset \mathcal{J}_{ma_{\max}}(\bfY)$ be the 
natural subvariety s.t. $\bfE_m(\F_{\ell^k})=E_m^k$.
More precisely, 
$$\bfE_m:=
\left\{(\sum_j  y_{1j} t^j,\dots, \sum_j y_{nj} t^j)\in \mathcal{J}_{ a_{\max} m}(\bfY) \,\,|\,\, y_{ij}=0 
\text{ for } 1 \leq i \leq n \text{ and } j<m a_{i}\right\}.$$

By the Lang-Weil bounds \cite[Theorem 1]{LW}, we obtain 
\begin{align*}    
\dim  \mathcal{J}_{ma_{\max}}(\gamma)^{-1}(\bfE_m)&=\limsup_{k\to \infty}\log_{\ell^k}(\#\mathcal{J}_{ma_{\max}}(\gamma)^{-1}(\bfE_m))(\F_{\ell^k}) \leq
\\&\leq
M+ma_{\max}\dim(\gamma^{-1}(0))+\limsup_{k\to \infty}\log_{\ell^k}(\#\bfE_m(\F_{\ell^k}))=
\\&=
M+ma_{\max}\dim(\gamma^{-1}(0))+\dim(\bfE_m)
\end{align*}

Note that   $\bfE_m\subset \mathcal{J}_{ma_{\max}}(\bfY)$ has an algebraic group structure coming from the additive structure on $\bfY$. 

The action of $\bG_m$ on $\bfY$ induce an action on  $\bfE_m$ and $\mathcal{J}_{ma_{\max}}(\bfY)$.
Let $\bfZ:=\mathcal{J}_{ma_{\max}}(\bfY)/\bfE_m$ and let $\eps:\mathcal{J}_{ma_{\max}}(\bfY)\to \bfZ$ be the quotient map. Consider the map $$\delta:=\eps \circ \mathcal{J}_{m a_{\max}}(\gamma): \mathcal{J}_{m a_{\max}}(\bfX)\to \bfZ.$$ We have $$\dim(\delta^{-1}(0))=\dim  \mathcal{J}_{m a_{\max}}(\gamma)^{-1}(\bfE_m)\leq M+m a_{\max}\dim(\gamma^{-1}(0))+\dim(\bfE_m)$$ 
By upper semi-continuity of the dimension of the fiber (\cite[IV 13.1.3,13.1.15]{EGA}), we have a Zariski open $\bfU\subset \bfZ$ s.t. for any $x\in \bfU(\bar\F_{\ell})$ we have:
\begin{equation}\label{eq:dim.in}
\dim(\delta^{-1}(x))\leq\dim(\delta^{-1}(0))\leq M+m a_{\max}\dim(\gamma^{-1}(0))+\dim(\bfE_m)    
\end{equation}
Using the action of $\bG_m$ on $\bfZ$ we obtain that 
\eqref{eq:dim.in} holds for any $x\in \bfZ(\bar\F_{\ell})$. 

Let 
$$\bfR_m:=\{(\sum_j  y_{1j} t^j,\dots, \sum_j y_{nj} t^j)\in \mathcal{J}_{ a_{\max} m}(\bfY)| y_{ij}=0 \text{ for } j<m a_{\min}\}.$$
Set $\bfW=\bfR_m/\bfE_m$ with the natural embedding to  ${\bf Z}$. 
Using 
\eqref{eq:dim.in} we deduce 
\begin{align*}
  \dim  \mathcal{J}_{m a_{\max}}(\gamma)^{-1}(\bfR_m)&=\dim(\delta^{-1}(\bfW))\leq
  \\&\leq
  M+m a_{\max}\dim(\gamma^{-1}(0))+\dim(\bfE_m)+\dim(\bfW)=
  \\&=
  M+m a_{\max}\dim(\gamma^{-1}(0))+\dim(\bfR_m)  
\end{align*}
Consider the commutative diagram.
$$
        \begin{tikzcd}
        \mathcal{J}_{m a_{\max}}(\bfX) 
        \dar[swap, "r_\bfX"] \rar["\mathcal{J}_{m a_{\max}}(\gamma)"] & \mathcal{J}_{m a_{\max}}(\bfY)\dar["r_\bfY"]  \\
        \mathcal{J}_{m a_{\min}}(\bfX) \rar["\mathcal{J}_{m a_{\min}}(\gamma)"] & \mathcal{J}_{m a_{\min}}(\bfY) 
        \end{tikzcd}
$$
were the vertical maps are the reduction maps.
We note that $\bfR_m=r_{\bfY}^{-1}(0).$
Clearly, 
$$
\dim(r_\bfY^{-1}(0)) -\dim(r_\bfX^{-1}(0))= m(a_{\rm \min}-a_{\rm \max})\dim(\gamma^{-1}(0))
$$
We have 
$$\mathcal{J}_{m a_{\max}}(\gamma)^{-1}(\bfR_m)=(r_\bfY\circ \mathcal{J}_{m a_{\max}}(\gamma)) ^{-1}(0)=(\mathcal{J}_{m a_{\min}}(\gamma) \circ r_\bfX)^{-1}(0)= r_\bfX^{-1}((\mathcal{J}_{m a_{\min}}(\gamma))^{-1}(0)).$$
Since $r_\bfX$ is a quotient map, we get 
\begin{align*}
  \dim (\mathcal{J}_{m a_{\min}}(\gamma)^{-1}(0))&= \dim(r_X^{-1}((\mathcal{J}_{m a_{\min}}(\gamma))^{-1}(0)))-\dim(r_X^{-1}(0))  
\\&=
\dim(\mathcal{J}_{m a_{\max}}(\gamma)^{-1}(\bfR_m))-\dim(r_\bfX^{-1}(0))  
\\&\leq
M+m a_{\max}\dim(\gamma^{-1}(0))+\dim(\bfR_m)-\dim(r_\bfX^{-1}(0))  
\\&=
M+m a_{\max}\dim(\gamma^{-1}(0))+\dim(r_\bfY^{-1}(0)) -\dim(r_\bfX^{-1}(0))  
\\&=
M+m a_{\min}\dim(\gamma^{-1}(0)) 
\end{align*}

Now, fix $m'$. Let $m$ be  the largest integer s.t. $m a_{\min}<m'$, that is $m=[\frac{m'}{a_{\min}}]$. Consider the restriction map $\mathcal{J}_{m'}(\gamma^{-1}(0))\to  \mathcal{J}_{m a_{\min}}(\gamma^{-1}(0))$. Its fibers are of dimension $\leq (m'-m a_{\min})\dim \bfX\leq  a_{\min}\dim \bfX$. We obtain 
\begin{align*}
\dim \mathcal{J}_{m'}(\gamma^{-1}(0)) &\leq   a_{\min}\dim \bfX+ \dim \mathcal{J}_{m a_{\min}}(\gamma^{-1}(0))=  a_{\min}\dim \bfX+ \dim \mathcal{J}_{m a_{\min}}(\gamma)^{-1}(0)\leq 
\\&\leq 
 a_{\min}\dim \bfX+ M+m a_{\min}\dim(\gamma^{-1}(0)) \leq
 \\&\leq 
 a_{\min}\dim \bfX+ M+m'\dim(\gamma^{-1}(0))= 
C + m'\dim(\gamma^{-1}(0)).
\end{align*}
\end{proof}
\section{Proof of Theorems \ref{thm:jet.nil}, \ref{thm:jet.fib},     and \ref{thm:jet.fib2}}\label{sec:pfs}

\begin{proof}[Proof of \Cref{thm:jet.nil}]
According to \Cref{thm:an.jet} the Chevalley map $p:\fg \to \fc$ is effectively an-FRS over \Rami{the origin}. We apply \Cref{prop:eff.FRS.imp.jet} with 
$p$ and as  
$p^{-1}(0)=\bfN$,
we obtain the result.
\end{proof}

\begin{proof}[Proof of \Cref{thm:jet.fib} ]
Let $\Rami{C}_0$ be as in \Cref{thm:jet.nil}. Fix $k$. Take $\Rami{C}=\Rami{C}_0$.  
By semi-continuity of the dimension of the fiber (\cite[IV 13.1.3,13.1.15]{EGA}), we have a Zariski open $\bfU\subset \RamiA{\cJ}_m(\uc)$ s.t. for any $x\in \bfU(\bar\F_{\ell})$ we have:
\begin{equation}\label{eq:dim.fib}
\dim(\RamiA{\cJ}_m(p)^{-1}(x))\leq\dim(p^{-1}(0))\leq \Rami{C}+m \dim(p^{-1}(0)).
\end{equation}
Using the action of $\bG_m$ on $\RamiA{\cJ}_m(\uc)$ and $\RamiA{\cJ}_m(\ug)$ we obtain that 
\eqref{eq:dim.fib} holds for any $x\in \RamiA{\cJ}_m(\uc)(\bar\F_{\ell})$ as required.
\end{proof}

\begin{proof}[Proof of Theorem 
\ref{thm:jet.fib2}]
Let $\Rami{C}$ be as in \Cref{thm:jet.fib}. Fix $i$ and let
$\Rami{C}_i:=i\Rami{C}$.
Let $p_i:\gi \to \uc$ be the projection. 
We deduce that for  any $x\in \mathcal{J}_m(\uc)(\bar\F_{\ell})$ we have:
\begin{align*}
\dim(\mathcal{J}_m(p_i)^{-1}(x))&=i\dim(\mathcal{J}_m(p)^{-1}(x))\leq
\\&\leq
i\Rami{C}+m i\dim(p^{-1}(0))  =
\\&=
i\Rami{C}+m \dim(p_i^{-1}(0)) 
\end{align*}
Note that by \Cref{cor:pqflat} the map $p_i$ is flat.
Thus 
\begin{align*}
\dim(\mathcal{J}_m(\ug^{\times_{\uc}i}))&=
\dim(\mathcal{J}_m(p_i)^{-1}(x))+\dim(\mathcal{J}_m(\uc))\leq
\\&\leq
i\Rami{C}+m \dim(p_i^{-1}(0))+\dim(\mathcal{J}_m(\uc))=\Rami{C}_i+m\dim(\ug^{\times_{\uc}k}).
\end{align*}
\end{proof}

\section{Almost integrability}\label{sec:AlmInt}
In this section we study several versions of integrability of  an algebraic variety. We will prove that they are equivalent under the assumption of existence of a resolution (see \Cref{thm:eq.almost.int}). 

\begin{defn}\label[defn]{def:alm.int}
    Let $\bfX$ be  a  variety and $\bfU\subset \bfX^{sm}$ be an open subset. We say that $(\bfX,\bfU)$ is:
    \begin{enumerate}
        \item\label{def:alm.int:1} asymptotically almost integrable\index{asymptotically almost integrable} if for any:
        \begin{itemize}
            \item open affine  $\bfV\subset \bfX$
            \item top form $\omega$ on $\bfV^{sm}$ 
        \end{itemize}         
        there is $M\in \N$ s.t. for any $m'\in \N$  and any rectification of $\bfU\cap \bfV$ we have 
$$\lim_{k\to \infty}\frac{1}{\ell^{kM}}\int_{B_{m'}^{\bfU\cap \bfV,k} \cap \bfV(\F_{\ell^k}[[t]])} |\omega|=0.$$
        \item\label{def:alm.int:2} Geometrically almost integrable\index{geometrically almost integrable} if for any open affine (rectified) $\bfV\subset \bfX$ and any top form \RamiA{$\omega$} on $\bfV^{sm}$  there is a resolution of singularities $\tilde \bfV \to \bfV$, which is an isomorphism over $\bfV^{sm}$, s.t. that one can extend $\omega$ to a rational form on $\tilde \bfV$ whose poles form an SNC divisor with multiplicities 1.               
        \item\label{def:alm.int:3} Analytically almost integrable\index{analytically almost integrable} if for any 
        \begin{itemize}
            \item         
        open affine  $\bfV\subset \bfX$,
            \item         a rectification on $\bfV\cap \bfU$,
                \item         a top form $\omega$ on $\bfV^{sm}$,
                    \item         $k\in \N$,
        \end{itemize}
        there is $M\in \N$ s.t. for any  $m\in \N$ we have 
$$\int_{B_m^{\bfU\cap \bfV,k} \cap \bfV(\F_{\ell^k}[[t]])} |\omega|< M(m+1)^M.$$      
        \end{enumerate}  
\end{defn}
\begin{remark}
    The notion of geometrically almost integrable does not depend on $U$.
\end{remark}

\begin{theorem}\label[theorem]{thm:eq.almost.int}
    Assume that $\bfX$ has a strong resolution (See \S\ref{ssec:conv}\eqref{ssec:conv:strong}). Then TFAE:
    \begin{enumerate}
        \item\label{thm:eq.almost.int:1} $(\bfX,\bfU)$ is asymptotically almost integrable. 
        \item\label{thm:eq.almost.int:2} $(\bfX,\bfU)$ is geometrically almost integrable. 
        \item\label{thm:eq.almost.int:3} $(\bfX,\bfU)$ is analytically almost integrable.
    \end{enumerate}
\end{theorem}

\begin{proof}[Idea of the proof]
We assume WLOG $\bfV=\bfX$.
\begin{itemize}    \item[\eqref{thm:eq.almost.int:1} or \eqref{thm:eq.almost.int:3}   $\Rightarrow$
\eqref{thm:eq.almost.int:2}:] 
We are given a strong resolution of $\bfX$. We have to show that the poles of $\omega$ after the pull-back to this resolution are simple. We assume the contrary, and take a generic point $x$ on the divisor where $\omega$ have a non-simple pole. We replace the integral in \eqref{thm:eq.almost.int:1} or \eqref{thm:eq.almost.int:3} by an integral over a small ball around $x$. This is a smaller integral so the bound in \eqref{thm:eq.almost.int:1} or \eqref{thm:eq.almost.int:3}  still valid for it. On the other hand we can compute it using local coordinates near this point and using the knowledge on the pole of $\omega$. A simple computation contradicts the bound provided by \eqref{thm:eq.almost.int:1} or \eqref{thm:eq.almost.int:3}.
\item[\eqref{thm:eq.almost.int:2} 
$\Rightarrow$
\eqref{thm:eq.almost.int:1} and \eqref{thm:eq.almost.int:3}:]   
We are given a resolution where all the poles of $\omega$ are simple. The integral in \eqref{thm:eq.almost.int:1} and \eqref{thm:eq.almost.int:3} can be computed on this resolution. This can be done locally. The local computation is a computation of an integral of a monomial top form on $\bA^n$ over a ball in the complement to a union of coordinate hyperplanes. This computation gives the required bound. 
\end{itemize}    
\end{proof}
\begin{proof}
$ $
\begin{itemize}    \item[\eqref{thm:eq.almost.int:1}$\Rightarrow$\eqref{thm:eq.almost.int:2}:] 
Let $\gamma:\tilde \bfX\to \bfX$ be a strong resolution. Let $\bfV,\omega$ be as in \Cref{def:alm.int}\eqref{def:alm.int:2}. WLOG assume $\bfX=\bfV$. Let $\bfZ=\gamma^{-1}(\bfX\smallsetminus \bfU)$ and $\tilde \bfU:=\gamma^{-1}(\bfU) $.  Note that $\tilde \bfU \cong \bfU$.
Take $\tilde \bfV=\tilde \bfX$. 

Let $\bfZ'\subset \bfZ$ be the exceptional divisor. We have $\bfZ=\bfZ'\cup \bfZ''$ where $\bfZ',\bfZ''\subset \bfZ$ are closed and have no common components. Let $\tilde \omega:=\gamma^{*}(\omega)$ considered as a rational form on $\tilde \bfX$.

The support of the poles of $\tilde\omega$ is contained in $\bfZ'$ which is an SNC divisor. Assume for the sake of contradiction that not all the poles are simple. Let $\bfZ_0$ be a component of $\bfZ'$ where $\tilde\omega$ have a pole of multiplicity $f>1$. Let $Z'''\subset \tilde \bfX$ be (the closure of) the  zero-locus of $\tilde\omega$.  

Replacing $\ell$ with its power if necessary  we may assume that  there is $z\in \bfZ_0(\F_\ell)$ that is outside all the other components of $\bfZ$ and $\bfZ'''$.  

Let $\bfV'$ be an affine open neighborhood of $z$ that does not intersect the other components of $\bfZ$ and $\bfZ'''$. Choose a rectification of $\bfU$. Using  the identification $\tilde\bfU \cong \bfU$ we obtain a rectification of $\tilde\bfU$.  Choose a rectification of $\bfV'\smallsetminus \bfZ_0=\tilde\bfU\cap \bfV'$ 
s.t. for any $m,k\in\N$ we have 
$$B_m^{\tilde\bfU\cap \bfV',k} \subset B_m^{\tilde\bfU,k}.$$
Finally, choose an arbitrary rectification on $\bfV'$.

Since $\gamma$ is defined over $\F_\ell$, for any $k$  we have
$$\gamma (\tilde \bfX(\F_{\ell^k}[[t]]))\subset \bfX(\F_{\ell^k}[[t]]).$$

By the assumption, there is $M\in \N$ s.t. for any $m'\in \N$ we have 
$$\lim_{k\to \infty}\frac{1}{\ell^{kM}}\int_{B_{m'}^{\bfU,k} \cap \bfX(\F_{\ell^k}[[t]])} |\omega|=0.$$

For any $m'\in \N$,
we obtain:


\begin{align*}    
0
&\leq\limsup_{k\to \infty}\frac{1}{\ell^{kM}}\int_{B_{m'}^{\bfV'\smallsetminus \bfZ_0,k}\cap B^{\bfV',k}_{-1}(z)} |\tilde\omega|
\\&\leq
\limsup_{k\to \infty}\frac{1}{\ell^{kM}}\int_{B_{m'}^{\bfV'\smallsetminus \bfZ_0,k} \cap \tilde \bfX(\F_{\ell^k}[[t]])} |\tilde\omega|
=\limsup_{k\to \infty}\frac{1}{\ell^{kM}}\int_{B_{m'}^{\bfV'\cap \tilde\bfU,k} \cap \tilde \bfX(\F_{\ell^k}[[t]])} |\tilde\omega|
\\&\leq
\limsup_{k\to \infty}\frac{1}{\ell^{kM}}\int_{B_{m'}^{\tilde
\bfU,k} \cap\tilde\bfX(\F_{\ell^k}[[t]])} |\tilde\omega|
\leq
\limsup_{k\to \infty}\frac{1}{\ell^{kM}}\int_{B_{m'}^{
\bfU,k} \cap\bfX(\F_{\ell^k}[[t]])} |\omega|
=0.
\end{align*}

Thus for any $m'\in\N$ we have:
\begin{align*}    
\lim_{k\to \infty}\frac{1}{\ell^{kM}}\int_{B_{m'}^{\bfV'\smallsetminus \bfZ_0,k}\cap B^{\bfV',k}_{-1}(z)} |\tilde\omega|
=0.
\end{align*}
Note that this statement is independent of the rectification on $\bfV'$.

Since $\bfZ_0$ is smooth, we can choose an affine $\bfV''\subset \bfV'$ and an etale map $\phi:\bfV''\to \A^{\dim\bfX}$ s.t.
\begin{itemize}
    \item $z\in\bfV''(\F_\ell)$
    \item $\phi(z)=0$
    \item $\phi^{-1}(\A^{\dim \bfX-1})=\bfZ_0\cap \bfV''$
\end{itemize}
Let $x_1$ be the defining coordinate of 
$\A^{\dim \bfX-1}\subset \A^{\dim \bfX} $.
Note that $$\phi^*\left(\frac{dx_1\wedge\dots \wedge dx_2}{x_1^f}\right)=g\tilde\omega ,$$  where $g\in O^\times_{\bfV''}(\bfV'')$.

Choose an arbitrary simple rectification on $\bfV''$
and choose the rectification on $\bfV''\smallsetminus \bfZ_0$ obtained by the embedding $\bfV''\smallsetminus \bfZ_0 \to \bfV''\times \A^1$ given by $\phi^*(x_1)$. 

For any $m',k\in \N$  we have 
\begin{align*}    
\lim_{k\to \infty}\frac{1}{\ell^{kM}}\int_{B_{m'}^{\bfV''\smallsetminus \bfZ_0,k}\cap B^{\bfV'',k}_{-1}(z)} |\tilde\omega|
=0
\end{align*}
On the other hand,
for any $m',k\in \N$  we have \begin{align*}    \int_{B_{m'}^{\bfV''\smallsetminus \bfZ_0,k}\cap B^{\bfV'',k}_{-1}(z)} |\tilde\omega|&=
\int_{\{x\in B_{m'}^{\bfV'',k}| val(x_1(\phi(x)))\leq m'\}\cap B^{\bfV'',k}_{-1}(z)} |\tilde\omega|
\\&=
\int_{\{x\in B^{\bfV'',k}_{-1}(z)| val(x_1(\phi(x)))\leq m'\}} |\tilde\omega|\\&=
\int_{\{x\in B^{\bfV'',k}_{-1}(z)| val(x_1(\phi(x)))\leq m'\}} |g\tilde\omega|
\\&=
\int_{(t\F_{\ell^k}[[t]]\smallsetminus t^{m'+1}\F_{\ell^k}[[t]])\times (t\F_{\ell^k}[[t]])^{\dim \bfX-1}}\left |\frac{dx_1\wedge\dots\wedge dx_2}{x_1^f}\right|
\\&=
\left(
\int_{(t\F_{\ell^k}[[t]]\smallsetminus t^{m'+1}\F_{\ell^k}[[t]])}\left |\frac{dx_1}{x_1^f}\right|
\right)
\cdot
\mu^{\A^{\dim\bfX-1},k} (t\F_{\ell^k}[[t]])^{\dim \bfX-1}
\\&=
\frac{\ell^{k}-1}{\ell^{k}}\left(\sum_{i=1}^{m'} \ell^{ik(f-1)} 
\right)
\frac{1}{\ell^{k(\dim \bfX-1)}}
\\&=
\frac{\ell^{k}-1}{\ell^{k}}\frac{\ell^{(m'+1)k(f-1)}-\ell^{k(f-1)}}{\ell^{k(f-1)}-1}
\frac{1}{\ell^{k(\dim \bfX-1)}}
\\&=
\ell^{k(f-\dim\bfX-1)}
\frac{(\ell^{k}-1)(\ell^{m'k(f-1)}-1)}{\ell^{k(f-1)}-1}
\end{align*}

So, for $m'> M+\dim \bfX$
$$\infty=
\lim_{k\to \infty}\ell^{-kM}\ell^{k(f-\dim\bfX-1)}
\frac{(\ell^{k}-1)(\ell^{m'k(f-1)}-1)}{\ell^{k(f-1)}-1}
=
\lim_{k\to \infty}\frac{1}{\ell^{kM}}
\int_{B_n^{V'\smallsetminus Z_0,k}\cap B^{V',k}_{-1}} |\tilde\omega|=0$$

Contradiction.
\item[\eqref{thm:eq.almost.int:2}$\Rightarrow$\eqref{thm:eq.almost.int:3}:] 
 Let $\bfV,\omega,k$ be as in \Cref{def:alm.int}\eqref{def:alm.int:3}. Choose a rectification of $\bfU\cap \bfV$. WLOG assume $\bfX=\bfV$.  Let $\phi:\tilde \bfX:=\tilde \bfV\to \bfV=\bfX$ be a resolution of singularities as in \Cref{def:alm.int}\eqref{def:alm.int:2}.
 Let $\tilde \omega:=\phi^*(\omega)$. Let $\bfZ$ be the divisor of poles of $\tilde\omega$. Let $\tilde \bfU:=\phi^{-1}(\bfU)$. 
 Choose a rectification on $\bfU$. We have to show that there exists $M$ s.t. for any $m$ we have 
$$\int_{B_m^{\bfU,k} \cap \bfX(\F_{\ell^k}[[t]])} |\omega|< M(m+1)^M$$      
  By \Cref{lem:kot.bnd}, this statement does not depend on the  rectification on $\bfU$. By the valuative criterion we have $\tilde\bfX(\F_{\ell^k}[[t]])= \phi^{-1}(\bfX(\F_{\ell^k}[[t]]))$. Here $\phi$ is interpreted as a map $\tilde\bfX(\F_{\ell^k}((t))) \to\bfX(\F_{\ell^k}((t)))$.

  So it is enough to show that for some rectification of $\tilde\bfU$ there exists $M$ s.t. 
  for any $m$ we have 
$$\int_{B_m^{\tilde\bfU,k} \cap \tilde\bfX(\F_{\ell^k}[[t]])} |\tilde\omega|< M(m+1)^M$$     
Let $\bfW\subset \tilde\bfX$ be the regular locus of $\tilde\omega$. We have $\tilde\bfU\subset \bfW$. Thus, for any rectification of $\bfW$ we can choose a rectification of $\tilde\bfU$ s.t. $B_m^{\tilde\bfU,k}\subset B_m^{\bfW,k}$ for any $m,k\in \N$. Therefore,  
 it is enough to show that for some rectification of $\bfW$ there exists $M$ s.t. 
  for any $m$ we have 
$$\int_{B_m^{\bfW,k} \cap \tilde\bfX(\F_{\ell^k}[[t]])} |\tilde\omega|< M(m+1)^M$$  
Let $\tilde\bfX=\bigcup_{i=1}^{I}\bfV_i$ be an open affine cover such that there are:
\begin{itemize}
    \item etale maps $\gamma_i:\bfV_i\to \A^d$ \item monomial rational forms $\omega_i$ on $\A^d$ with powers $\geq -1$, \RamiB{and }
    \item regular functions $g\in O_{\tilde \bfX}(\bfV_i)$,
\end{itemize} satisfying $\tilde\omega|_{\bfV_i}=g_i\gamma_i^*(\omega_i)$.

It is enough to show that for any $i$ there is a rectification of $\bfV_i\cap \bfU$ and $M$ s.t.  
  for any $m$ we have 
$$\int_{B_m^{\bfV_i\cap \bfW,k} \cap \tilde\bfV_i(\F_{\ell^k}[[t]])} |\gamma_i^*(\omega_i)|< M(m+1)^M$$  
Let $\bfW_i$ be the regular locus of $\omega_i$. This is a complement to coordinate hyperplanes in $\A^d$. So it is equipped with the standard embedding into $\A^{d+1}$ which gives us a simple rectification on $\bfW_i$. We can find a section on $\bfV_i\cap \bfW$ s.t. $B_m^{\bfV_i\cap \bfW,k}\subset \gamma_i^{-1}(B_m^{\bfW_i,k})$ for any $m,k\in \N$.

Let $A_i\in\N$ s.t.  for any field $F$ and any $x\in F^d$ we have $\#\gamma_i^{-1}(x)<A$. For any $k,m$ we obtain:
$$\int_{B_m^{\bfV_i\cap \bfW,k} \cap \tilde\bfV_i(\F_{\ell^k}[[t]])} |\gamma_i^*(\omega_i)|
\leq 
\int_{\gamma_i^{-1}(B_m^{\bfW_i,k} \cap \F_{\ell^k}[[t]]^d)} |\gamma_i^*(\omega_i)|
\leq 
A\int_{B_m^{\bfW_i,k} \cap \F_{\ell^k}[[t]]^d} |\omega_i|
$$
So,
it is enough to show that for any $i$ there is $M\in \N$ s.t.  
  for any $m$ we have 
$$\int_{B_m^{\bfW_i,k} \cap \F_{\ell^k}[[t]]^d} |\omega_i|< M(m+1)^M.$$  This is a straightforward computation.
\item[\eqref{thm:eq.almost.int:3}$\Rightarrow$\eqref{thm:eq.almost.int:2}:] 
The proof is similar to the implication \eqref{thm:eq.almost.int:1}$\Rightarrow$\eqref{thm:eq.almost.int:2} and we will not use this implication.
\item[\eqref{thm:eq.almost.int:2}$\Rightarrow$\eqref{thm:eq.almost.int:1}:] 
The proof is similar to the implication \eqref{thm:eq.almost.int:2}$\Rightarrow$\eqref{thm:eq.almost.int:3} and we will not use this implication.
\end{itemize}
\end{proof}

\section{Almost integrability of $\gi$}\label{sec:AlmIntgi}
In this section we prove the following:
\begin{theorem}\label[theorem]{thm:alm.int.fib}
Let $i\in\N$ be an integer. 
Let $\bfU_i \subset \gi$ be the preimage of $\uc^{rss}$.
Then
   $(\gi,\bfU_i)$  is asymptotically almost integrable.
\end{theorem}
\begin{proof}[Idea of the proof]
    We use the bound on the dimension of the jet schemes of $\ug_i$ (see \Cref{thm:jet.fib2}). This bound gives us a bound on $\#\mathcal{J}_m(\ug_i)(\F_{\ell^k})$ for large $k$. This bounds the $L^i$ norm of $p_*(\mu_0^{\ug,k})*1_{B_m^{\ug,k}}.$ 
    Since $p$ is smooth over $\uc^{rss}$, the measure  $p_*(\mu_0^{\ug,k})$ is $m$-smooth in large  balls in $\uc^{rss}$. Thus we get a bound on the $L^i$ norm of $p_*(\mu_0^{\ug,k})$ over large ball in $\uc^{rss}$. This bounds the volume of a large ball in $\bfU_i$ as required.
\end{proof}

\begin{proof}
Let \RamiG{$C_i$} be as in \Cref{thm:jet.fib2}, let \RamiG{$c:=C_i$} and let $M:=c+1$. Fix $m\in \N$.
Let $\omega_\ug$ be the standard top form on $\ug$  (coming from the identification $\ug=\bA^{n^2}$) and $\omega_\uc$ be the standard top form on $\uc$ (coming from the identification $\uc=\bA^n$). 
Let  $\omega_{\gi}:=\omega_\ug^{\times_{\omega_{\uc} i}}$. This is a rational top form on $\gi$. 
Recall that $\ug^r$ is the smooth locus of $p_0:\ug\to\uc$. Thus
$\omega_{\gi}$ is regular and invertible on $(\ug^r)^{\times_\uc i}$. By \Cref{cor:pFibReduced}, $\ug^r$ is big in $\ug$. Since  $p$ is flat (see \Cref{cor:pqflat}), this implies that      $(\ug^r)^{\times_\uc i}$ is big in $\gi$. Therefore $\omega_{\gi}$ is regular  and invertible on $\gi^{sm}$. Let $p^i:\gi\to \uc$ be the projection. Choose the standard \RamiA{$\mu$-rectification} of $\uc$ and $\ug$. 
Choose the simple \RamiA{$\mu$-rectification} on $\uc^{rss}$ given by the embedding $\uc^{rss}\to \uc \times \A^1$ using the discriminant and the top-form induced from $\uc$. This gives  simple \RamiA{$\mu$-rectifications} on $\ug^{rss}$ and $\bfU_i$.

By \Cref{lem:push.mes.sm} there exist $\RamiG{m''>}m'\RamiG{>m}$ s.t. for any $k$ there is an $\RamiG{m''}$-smooth function $f_k$ on $C^\infty(B_{\RamiG{m'}}^{\uc,k})$ s.t. $$(p_0)_*(1_{B_{0}^{\ug,k}}\mu_{m}^{\ug^{rss},k})=f_k \mu_{\RamiG{m'}}^{\uc,k}.$$
\RamiG{Since $\supp((p_0)_*(1_{B_{0}^{\ug,k}}\mu_{m}^{\ug^{rss},k}))\subset B_{0}^{\fc,k}$ we can assume that $\supp(f_k)\subset B_{0}^{\fc,k}$. Hence
$$(p_0)_*(1_{B_{0}^{\ug,k}}\mu_{m}^{\ug^{rss},k})=f_k \mu_{\RamiG{m}}^{\uc,k}.$$}
\newpage
    We get
\begin{align}\label{thm:alm.int.fib:long}
\int_{B_m^{\bfU_i,k} \cap B_0^{\gi,k}}|\omega_{\gi}|&=
    \int_{B_m^{\uc^{rss},k} \cap B_0^{\uc,k}} \left(\frac{p_*\left(|\omega_{\RamiG{\ug}}|1_{B_0^{\ug,k}}\right)}{|\omega_{\RamiG{\uc}}|}\right)^i|\omega_{\RamiG{\uc}}|
 \nonumber   \\&=
    \int_{B_m^{\uc^{rss},k} \cap B_0^{\uc,k}} \left(\frac{p_*\left(1_{B_{0}^{\ug,k}}\mu_{m}^{\ug^{rss},k}\right)}{|\omega_{\RamiG{\uc}}|}\right)^i|\omega_{\RamiG{\uc}}|
\nonumber    \\&=
    \int_{B_m^{\uc^{rss},k} \cap B_0^{\uc,k}} \left(f_k\right)^i|\omega_{\RamiG{\uc}}|
\nonumber    \\&=
    \frac{1}{(|\omega_\uc|({B_{-\RamiG{m''}}^{\uc,k}}))^{\RamiG{i}}}\int_{B_m^{\uc^{rss},k} \cap B_0^{\uc,k}} \left(f_k*(|\omega_\uc|1_{B_{-\RamiG{m''}}^{\uc,k}})\right)^i|\omega_{\RamiG{\uc}}|
\nonumber    \\&=
    \ell^{n\RamiG{m''}k\RamiG{i}}\int_{B_m^{\uc^{rss},k} \cap B_0^{\uc,k}} \left(f_k*(|\omega_\uc|1_{B_{-\RamiG{m''}}^{\uc,k}})\right)^i|\omega_{\RamiG{\uc}}|
\nonumber    \\&=
    \ell^{n\RamiG{m''}k\RamiG{i} }\int_{B_m^{\uc^{rss},k} \cap B_0^{\uc,k}} \left(\frac{p_*\left(1_{B_{0}^{\ug,k}}\mu_{m}^{\ug^{rss},k}\right)}{|\omega_{\RamiG{\uc}}|}*(|\omega_\uc|1_{B_{-\RamiG{m''}}^{\uc,k}})\right)^i|\omega_{\RamiG{\uc}}|
   \\&=
    \ell^{n\RamiG{m''}k\RamiG{i}}\int_{B_m^{\uc^{rss},k} \cap B_0^{\uc,k}} \left({p_*\left(1_{B_{0}^{\ug,k}}\mu_{m}^{\ug^{rss},k}\right)}*(1_{B_{-\RamiG{m''}}^{\uc,k}})\right)^i|\omega_{\RamiG{\uc}}|
\nonumber    \\&=
    \ell^{n\RamiG{m''}k\RamiG{i}}\int_{B_m^{\uc^{rss},k} \cap B_0^{\uc,k}} \left(p_*\left(|\omega_{\RamiG{\ug}}|1_{B_0^{\ug,k}}\right)*(1_{B_{-\RamiG{m''}}^{\uc,k}})\right)^i|\omega_{\RamiG{\uc}}|
\nonumber    \\&<
    \ell^{n\RamiG{m''}k\RamiG{i}}\int_{B_0^{\uc,k}} \left(p_*\left(|\omega_{\RamiG{\ug}}|1_{B_0^{\ug,k}}\right)*(1_{B_{-\RamiG{m''}}^{\uc,k}})\right)^i|\omega_{\RamiG{\uc}}|
 \nonumber   \\&=
    \ell^{n\RamiG{m''}k\RamiG{i}}\sum_{B\in  B_0^{\uc,k}/B_{-\RamiG{m''}}^{\uc,k}}\int_{B} \left(p_*\left(|\omega_{\RamiG{\ug}}|1_{B_0^{\ug,k}}\right)*(1_{B_{-\RamiG{m''}}^{\uc,k}})\right)^i|\omega_{\RamiG{\uc}}|
\nonumber    \\&=
    \ell^{n\RamiG{m''}k\RamiG{i}}\sum_{B\in  B_0^{\uc,k}/B_{-\RamiG{m''}}^{\uc,k}}
    \left({|\omega_{\ug}|(p^{-1}(B)\cap B_0^{\ug,k})}
    \right)^i |\omega_{\uc}|(B)    
\nonumber    \\&=
    \ell^{(i-1)n\RamiG{m''}k}
    \sum_{B\in  B_0^{\uc,k}/B_{-\RamiG{m''}}^{\uc,k}}
    \left({|\omega_{\ug}|(p^{-1}(B)\cap B_0^{\ug,k})}\right)^i 
\nonumber    \\&=
    \ell^{(i-1)n\RamiG{m''}k}
    \sum_{x\in \mathcal{J}_{\RamiG{m''}}(\uc)(\F_{\ell^k})}
    \left(\#\left((\mathcal{J}_{\RamiG{m''}}(p))^{-1}(x)\right) |\omega_{\RamiG{\ug}}|(B_{-\RamiG{m''}}^{\ug,k})\right)^i 
\nonumber    \\&=
    \ell^{(i-1)n\RamiG{m''}k-i\RamiG{m''}n^2k}
    \sum_{x\in \mathcal{J}_{\RamiG{m''}}(\uc)(\F_{\ell^k})}
    \left(\#\left((\mathcal{J}_{\RamiG{m''}}(p))^{-1}(x)\right) \right)^i 
\nonumber    \\\nonumber &=
    \ell^{(i-1)n\RamiG{m''}k-i\RamiG{m''}n^2k}
    \#\mathcal{J}_{\RamiG{m''}}(\gi)(\F_{\ell^k})=\ell^{-\RamiG{m''}k\dim\gi}
    \#\mathcal{J}_{\RamiG{m''}}(\gi)(\F_{\ell^k}).    
\end{align}

Therefore 

\begin{align*}
0&\leq \lim_{k\to \infty}\frac{1}{\ell^{kM}}\int_{B_m^{\bfU_i,k} \cap B_0^{\gi,k}}|\omega_{\gi}|
\leq 
\lim_{k\to \infty}\frac{\ell^{-\RamiG{m''}k\dim\gi}
    \#\mathcal{J}_{\RamiG{m''}}(\gi)(\F_{\ell^k})}{\ell^{kM}} 
\\&=
    \lim_{k\to \infty}\frac{
    \#\mathcal{J}_{\RamiG{m''}}(\gi)(\F_{\ell^k})}{\ell^{k(M+\RamiG{m''}\dim\gi)}}
=
    \lim_{k\to \infty}\frac{
    \#\mathcal{J}_{\RamiG{m''}}(\gi)(\F_{\ell^k})}{\ell^{k(c+1+\RamiG{m''}\dim\gi)}}
\\&\leq
    \lim_{k\to \infty}\frac{
    \#\mathcal{J}_{\RamiG{m''}}(\gi)(\F_{\ell^k})}{\ell^{k(1+\dim \mathcal{J}_{\RamiG{m''}}(\gi))}}
=
    \lim_{k\to \infty}\ell^{-k}\frac{
    \#\mathcal{J}_{\RamiG{m''}}(\gi)(\F_{\ell^k})}{\ell^{k(\dim \mathcal{J}_{\RamiG{m''}}(\gi))}}=0.
\end{align*}
 The last equality follows from the Lang-Weil bound.   
So
$$ \lim_{k\to \infty}\frac{1}{\ell^{kM}}\int_{B_m^{\bfU_i,k} \cap B_0^{\gi,k}}|\omega_{\gi}|=0,$$
as required.
\end{proof}
\section{Proof of \Cref{thm:alm.an.frs}}\label{sec:Pf.alm.an.frs}
We will need the following:
\begin{lemma}\label{lem:alm.int.implies.alm.Linf}
    Let $\phi:\bfX\to \bfY$ be a flat map of smooth (\RamiA{$\mu$-rectified}) varieties. Assume that the smooth locus of $\phi$ is big in $\bfX$. Let $\bfU\subset \bfY$ be an open dense subset of the locus of regular values of $\phi$.  Let $i,k\in \N$. Let $\phi_i:\bfX_i:=\bfX^{\times_\bfY i}\to \bfY$ be the projection.  Let $\bfV =\phi_i^{-1}(\bfU)$.

    Assume that $(\bfX_i,\bfV)$ is analytically almost integrable.
 
 Then $\phi_*(\mu_0^{\bfX,k})\in L^{i'}(B^{\bfY,k}_\infty)$ for any $i'<i$.
\end{lemma}
\begin{proof}
WLOG assume that the \RamiA{$\mu$-rectification}s of $\bfX$ and $\bfY$ are simple.
\begin{enumerate}[Step 1.]
    \item  $\exists M>0 \text{ s.t. } \forall m\in \N \text{ we have }\int_{B^{\bfU,k}_m\cap B^{\bfY,k}_0} \left(\frac{\phi_*(\mu_0^{\bfX,k})}{\mu^{\bfY,k}} \right)^i \mu^{\bfY,k}< M m^M $.\\
    Using the  embedding $\bfX_i \to  \bfX^i$ we obtain an embedding of $\bfX_i$ into an affine space. Let $\omega_\bfX$, $\omega_\bfY$ be the forms on $\bfX$ and $\bfY$ and let $\omega_{\bfX_i}:=\omega_{\bfX}^{\times_{\omega_\bfY}i}$. This is a rational form. The condition implies that it is regular on a big subset of $\bfX_i$ and hence can be extended to the smooth locus of $\bfX_i$. This together with the embedding above gives us a \RamiA{$\mu$-rectification} of $\bfX_i$. Choose a \RamiA{$\mu$-rectification} on $\bfU$, and choose a rectification of $\bfV$ s.t. for any $m\in \N$ we have $$B_m^{\bfV,k}=B_m^{\bfX_i,k}\cap \phi_i^{-1}(B_m^{\bfU,k}).$$    
    Since $(\bfX_i,\bfV)$ is analytically almost integrable we can find a constant $M\in \N$ s.t. for any  $m\in \N$ we have 
$$\int_{B_m^{\bfV,k} \cap B_0^{\bfX_i,k}} |\omega_{\bfX_i}|< Mm^M.$$      
Let $m\in \N.$ We obtain:
\begin{align*}    
\int_{B^{\bfU,k}_m\cap B^{\bfY,k}_0} \left(\frac{\phi_*(\mu_0^{\bfX,k})}{\mu^{\bfY,k}} \right)^i \mu^{\bfY,k}&=\int_{\phi_i^{-1}(B^{\bfU,k}_m\cap B^{\bfY,k}_0) \cap B_0^{\bfX_i,k}} |\omega_{\bfX_i}|
\\&
=\int_{\phi_i^{-1}(B^{\bfU,k}_m)\cap \phi_i^{-1}(B^{\bfY,k}_0) \cap B_0^{\bfX_i,k}} |\omega_{\bfX_i}|
\\&
=\int_{(B^{\bfV,k}_m)\cap \phi^{-1}(B^{\bfY,k}_0) \cap B_0^{\bfX_i,k}} |\omega_{\bfX_i}|
\\&
=\int_{B_m^{\bfV,k} \cap B_0^{\bfX_i,k}} |\omega_{\bfX_i}| < M m^M 
\end{align*}

    \item there is $0<g\in L_{\loc}^{<\infty}(B^{\bfY,k}_\infty)$ such that 
    $\frac{\phi_{*}(\mu_0^{\bfX,k})}{g\mu^{\bfY,k}}\in L^{i}(B^{\bfY,k}_\infty)$.\\
    Take $$g(y):=
    \begin{cases}
			\min\{m\geq 1|y\in B_m^{\bfU,k}\}^{M+2}, & \text{ if } y\in B^{\bfU,k}_\infty\\
            0, & \text{otherwise}
		 \end{cases}
    $$
For any $m\geq 1$ we have:
\begin{align*}    
\int_{(B^{\bfU,k}_{m}\smallsetminus B^{\bfU,k}_{m-1}) \cap B_0^{\bfY,k}} \left(\frac{\phi_*(\mu_0^{\bfX,k})}{g\mu^{\bfY,k}} \right)^i \mu_Y&=
\int_{(B^{\bfU,k}_{m}\smallsetminus B^{\bfU,k}_{m-1}) \cap B_0^{\bfY,k}} \left(\frac{\phi_*(\mu_0^{\bfX,k})}{m^{M+2}\mu^{\bfY,k}} \right)^i \mu_Y
\\&=
\frac{1}{m^{(M+2)i}}\int_{(B^{\bfU,k}_{m}\smallsetminus B^{\bfU,k}_{m-1}) \cap B_0^{\bfY,k}} \left(\frac{\phi_*(\mu_0^{\bfX,k})}{\mu^{\bfY,k}} \right)^i 
\\&<
\frac{1}{m^{(M+2)i}}\int_{(B^{\bfU,k}_{m}) \cap B_0^{\bfY,k}} \left(\frac{\phi_*(\mu_0^{\bfX,k})}{\mu^{\bfY,k}} \right)^i 
\\&<
\frac{M m^M}{m^{(M+2)k}}\leq \frac{M}{m^2} 
\end{align*}

Thus
\begin{align*}    
\int_{B^{\bfY,k}_\infty} \left(\frac{\mu_0^{\bfX,k}}{g\mu^{\bfY,k}} \right)^i \mu_\bfY&=\int_{B^{\bfU,k}_\infty} \left(\frac{\mu_0^{\bfX,k}}{g\mu^{\bfY,k}} \right)^i \mu_\bfY
\\&=
\int_{B^{\bfU,k}_0} \left(\frac{\mu_0^{\bfX,k}}{g\mu^{\bfY,k}} \right)^i \mu_\bfY
+\sum_{m=1}^{\infty}
\int_{B^{\bfU,k}_m\smallsetminus B^{\bfU,k}_{m-1}} \left(\frac{\mu_0^{\bfX,k}}{g\mu^{\bfY,k}} \right)^i \mu_\bfY
\\&<
M+ \sum_{m=1}^\infty\frac{M}{m^2} <3M.
\end{align*}    
It remains to show that $g\in L_{\loc}^{<\infty}(B^{\bfY,k}_\infty)$. It is easy to see that $g|_{B_\infty^{\bfU,k}}$ is  a norm function as described in \S\ref{sec:norm}. Thus by \Cref{prop:log.norm.in.L1} $g\in L_{\loc}^{<\infty}(B^{\bfY,k}_\infty)$.
    \item End of the proof.\\
Let $i'<i$ and let $A>1$ be such that 
$\frac{1}{i}+\frac{1}{A}=\frac{1}{i'}.$
By the previous step 
$g\in L_{\loc}^{A}(B^{\bfY,k}_\infty)$ 
and   $\frac{\phi_{*}(\mu_0^{\bfX,k})}{g\mu^{\bfY,k}}\in L_{\loc}^{i}(B^{\bfY,k}_\infty)$. 

Let $m\in \Z_{>0}$. We obtain:
\begin{itemize}
    \item $g 1_{B^{\bfY,k}_m}\in L^{A}(B^{\bfY,k}_m)$ 
    \item $\frac{\phi_{*}(\mu_0^{\bfX,k})}{g\mu^{\bfY,k}} 1_{B^{\bfY,k}_m}\in L^{i}(B^{\bfY,k}_m)$. 
\end{itemize}

We recall the generalized H{\"o}lder inequality. 
Suppose that:
\begin{itemize}
    \item 
$r_1,r_2,r$ are positive and satisfy $\frac{1}{r}=\frac{1}{r_1}+\frac{1}{r_2}$ \item $f_1 \in L^{r_1}(B^{\bfY,k}_m), f_2 \in L^{r_2}(B^{\bfY,k}_m).$
\end{itemize}
Then 
$$
||f_1f_2||_{r} \leq ||f_{1}||_{r_1}||f_{2}||_{r_2} 
$$
See \cite
[Chapter 8, Exercise 6]
{WZ}.

Thus we have:
$$
||1_{B^{\bfY,k}_m}\frac{\phi_{*}(\mu_0^{\bfX,k})}{\mu^{\bfY,k}}||_{i'}=
||1_{B^{\bfY,k}_m}g \frac{\phi_{*}(\mu_0^{\bfX,k})}{g\mu^{\bfY,k}}||_{i'} \leq ||1_{B^{\bfY,k}_m} g||_{A}|| 1_{B^{\bfY,k}_m}  \frac{\phi_{*}(\mu_0^{\bfX,k})}{g\mu^{\bfY,k}}||_{i}< \infty,
$$
as requested.
    
\end{enumerate}    
\end{proof}
\begin{proof}[Proof of \Cref{thm:alm.an.frs}]
Using the homothety action of $\bG_m$ on $\ug$ we can assume, WLOG, that $\mu \leq c \mu_0^{\ug,1}$ for some constant $c>0$. Hence we can assume, WLOG, that $\mu = \mu_0^{\ug,1}$.

Let $\bfU_i \subset \gi$ be the preimage of $\uc^{rss}$ as in \Cref{thm:alm.int.fib}. 

By  \Cref{thm:alm.int.fib}
   $(\gi,\bfU_i)$  is asymptotically almost integrable.
By Theorem  \ref{thm:eq.almost.int}, this together with the assumption implies that  $(\gi,\bfU_i)$  is analytically almost integrable. By \Cref{lem:alm.int.implies.alm.Linf}, applied to the map $p:\ug\to\uc$,  this implies the assertion. Note that the rest of the properties of $p$  required by \Cref{lem:alm.int.implies.alm.Linf} are verified in \S\ref{sec:basic}.
\end{proof}

\section{Alternative versions of \Cref{thm:alm.an.frs}}\label{sec:alt.form}
In \Cref{thm:alm.an.frs} one can replace the condition of existence of strong resolution of \RamiA{$\ug_i$} with one of the following 2 conditions:
\begin{enumerate}[(a)]
    \item\label{sec:alt.form:1} The defining ideal of $\ug_i$ inside $\ug^{\RamiA{\times}i}$ has monomial  principalization (see \Cref{def.mon.princp} below).
    \item \label{sec:alt.form:2}
    \EitanA{\Cref{conj:cnt} below is valid for the variety $\ug_i$}.
\end{enumerate}
\begin{definition}[Monomial Principalization]\label{def.mon.princp}
$ $
\begin{enumerate}
    \item An ideal sheaf $I$ on a smooth variety $\bfX$ is called locally monomially principal\index{locally monomially principal} if there exist
    \begin{itemize}
        \item a finite Zariski open affine cover $\bfX=\bigcup \bfU_i$,
        \item a collection of etale maps $\phi_i:\bfU_i \to \A^{d_i}$, and
        \item monomials $f_i$ on $\A^{d_i}$
    \end{itemize}
    such that  for any $i$ we have $I|_{\bfU_i} =\phi_i^*(f_i)O_{\bfU_i}$.
    \item Given an ideal sheaf $I$ on a smooth variety $\bfX$, we say that \Dima{a map} $\phi :\tilde \bfX \to  \bfX$  is a monomial principalization\index{monomial principalization}  of $I$
if it satisfies:
\begin{itemize}
    \item $\tilde \bfX$ is smooth and $\phi$ is a proper birational map (a modification).
    \item The pulled-back ideal sheaf $\phi^{*}(I)$ is locally monomially principal.
\end{itemize}
\end{enumerate}
\end{definition}
\EitanA{
\begin{conj}\label{conj:cnt}
    Let $\bfZ$ be a variety. Assume that there is a constant $C$ s.t. 
\begin{equation}\label{conj:cnt:1}
  \forall m\geq 0 \text{ we have } \dim (J_m(\bfZ))\leq m\dim(\bfZ)+C.  
\end{equation}

Then there exists a constant $D$ s.t.  
\begin{equation}\label{conj:cnt:2}
\forall m\geq 0 \text{ we have } \#J_m(\bfZ)(\F_\ell)\leq Dm^D  \ell^{m\dim(\bfZ)}.
\end{equation}

\end{conj}
}
Since the difference between these 2 versions of \Cref{thm:alm.an.frs} \EitanA{and the original formulation of \Cref{thm:alm.an.frs}} is technical and it is not clear how useful \EitanA{these versions are} going to be, we will not give complete proofs of these versions but only \Dima{sketches}.

\EitanA{
\subsection{Sketch of the proof of \Cref{thm:alm.an.frs} with condition of strong resolution replaced with Condition \eqref{sec:alt.form:2}}

\Cref{thm:jet.fib2} implies \Dima{condition \eqref{conj:cnt:1}} of \Cref{conj:cnt} for the variety $\ug_i$. Thus, assuming  \Cref{conj:cnt}  for $\ug_i$, we obtain \eqref{conj:cnt:2}  for this variety.
Similarly to the proof of  
\Cref{thm:alm.int.fib} (specifically, using the inequality \eqref{thm:alm.int.fib:long}), this implies that the pair $(\ug_i,\bfU_i)$ is analytically almost integrable, where $\bfU_i$ is as in \Cref{thm:alm.int.fib}.

The rest of the proof is the same as the proof of \Cref{thm:alm.an.frs}.
}

\subsection{Sketch of the proof of \Cref{thm:alm.an.frs} with condition of strong resolution replaced with Condition \eqref{sec:alt.form:1}}
The proof is based on the following version of \cite[Theorem 3.1]{Mu}
\begin{theorem}\label{thm:Mus}
Let $\bfX\subset \bfM$ be a closed subvariety of a smooth algebraic variety. Let $\gamma:\tilde\bfM\to \bfM$ be a principalization of the defining ideal sheaf $\cI_\bfX$ of $\bfX$ inside $\bfM$ Let $a_i,b_i$ be as in \cite[Theorem 3.1]{Mu}. That is $a_i$ are the coefficients of the components of the pre-image of  $\cI_\bfX$ and $b_i$ are the coefficients of the components of the discrepancy divisor of $\gamma$.

Then, the following statements are equivalent:
\begin{enumerate}[(i)]
    \item For every $i \geq 1$, we have $b_i \geq (\dim(\bfM)-\dim(\bfX)) a_i - 1$.
    \item There is a constant $c$ s.t. $\dim \mathcal{J}_m(\bfX) \leq (m + 1) \dim \bfX+c$, for every $m \geq 1$.
    \item There exists $M\in \N$ s.t. for any $m,k\in\N$ we have $$\#\mathcal{J}_m(\bfX)(\F_{\ell^k})< M m^M \ell^{km\dim \bfX}.$$ 
\end{enumerate}
\end{theorem}
The proof of this theorem is parallel to the original proof of \cite[Theorem 3.1]{Mu}. One also has to use the Denef's formula, but one can use the usual Denef formula and not the motivic one.

\EitanA{This theorem, together with condition \eqref{sec:alt.form:1}, implies condition \eqref{sec:alt.form:2}. Thus we are done.}

\section{Proof of \Cref{thm:uncond.an.frs}}\label{sec:PfUncond}
\subsection{an-FRS maps of analytic varieties}
Let $F:=\F_\ell((t))$.\index{$F$} Note that we can extend the notion of an-FRS morphisms (see \Cref{def:FRS} \eqref{def:FRS:2}) to the set-up of morphisms between $F$-analytic varieties in a natural way.

We require the following straightforward lemma about properties of an-FRS maps.

\begin{lemma}\label{lemma:an-FRS} 
 Let $\gamma:X \to Y$ be a morphism of smooth $F$-analytic varieties.     
    Let $\delta_i:\tilde X_i \to X$ for $i\in I$ be a (possibly infinite) collection of submersions. Then:
    \begin{enumerate}
        \item if $\gamma$ is an-FRS then, for each $i$,  $\gamma \circ \delta_i$ is an-FRS.
        \item if $\bigcup_{i\in I} \delta_{i}(\tilde X_i)=X$  and  $\gamma \circ \delta_i$ is  an-FRS for all $i$,  then $\gamma$ is an-FRS.
    \end{enumerate}

\end{lemma}

\subsection{Harish-Chandra descent for an-FRS maps}
\begin{notation} $ $
    Denote 
    \begin{itemize}
        \item $\g:=\ug(F)$\index{$\g$}
        \item $\g^{ss}:=$\index{$\g^{ss}$} the collection of semi-simple elements in $\g$, i.e. those elements with a separable minimal polynomial.
        \item $\bfN_{insep}\subset \ug$\index{$\bfN_{insep}$}  - be the locus of all matrices whose characteristic polynomial is totally inseparable (i.e. has a single root over the algebraic closure). Note that it has a natural structure of an algebraic variety.
        \item $N_{insep}:=\bfN_{insep}(F)$
        \index{$N_{insep}$}
        \item $\bfN_{insep,s}:=\bfN_{insep}\smallsetminus \ug^r$ where $\ug^r$ is the smooth locus of $p$.\index{$\bfN_{insep,s}$}
        \item $G:=\bfG(F)$\index{$G$}
        \item For $x\in \g^{ss}$ denote by $\alpha_x:G\times \g_x\to \g$\index{$\alpha_x$} be the action map
        \item For a map $\alpha:X\to Y$ of smooth $F$-analytic varieties denote by $\alpha^{reg}\subset X$ to be the collection of regular points.\index{$\alpha^{reg}$}
        \item $\fz:=\uz(F)$\index{$\fz$}
        \item $N:=\bfN(F)$\index{$N$}
    \end{itemize}
\end{notation}

%
\begin{lemma}\label{lem:HC.des.geo}
 $\bigcup_{x\in \g^{ss}\smallsetminus\fz }\alpha_x(\alpha_x^{reg})= \fg\smallsetminus N_{insep}$.
\end{lemma}
\begin{proof}
    Let $A:=\bigcup_{x\in \g^{ss}\smallsetminus\fz }\alpha_x(\alpha_x^{reg})$ take $y\in \fg\smallsetminus N_{insep}$ we have to show that $y\in A$.
    \begin{enumerate}[{Case} 1:]
        \item\label{lem:HC.des.geo:1} The minimal polynomial of $y$ is irreducible.\\
        Let $f$ be the minimal polynomial of $y$. we can write $f(s)=g(s^{p^k})$ where $g$ is separable polynomial. Let $x:=y^{p^k}$.  It is left to show that $(1,y)$ is a regular point of $\alpha_x$. In other words we have to show that $\g_x +[\g,y]=\g$. Passing to the orthogonal compliment w.r.t. the trace form, we need to show that $\g_x^\bot \cap [\g,y]^\bot=0$. Now:
        $$\g_x^\bot \cap [\g,y]^\bot= [\g,x] \cap \g_y\subset [\g,x] \cap \g_x =0$$
        \item \label{lem:HC.des.geo:2}  The minimal polynomial of $y$ is a power of an irreducible polynomial.\\
        Let $f=g^k$ be the minimal polynomial of $y$, with $g$ being irreducible.
        We note that using rational canonical form we can find $x\in \g$   s.t. $\overline{Ad(G)\cdot y}\ni x$ and the minimal polynomial of $x$ is $g$. By the previous case, $x\in A$. Since $A$ is open we are done.
        \item \label{lem:HC.des.geo:3}  The minimal polynomial of $y$ is product of 2 co-prime polynomials.\\
        In this case we can use the Primary decomposition theorem (from linear algebra) to find $x\in \g^{ss}\smallsetminus \fz$ s.t. $\g_x\supset \g_y$. Now the claim is proven as in Case \ref{lem:HC.des.geo:1}.
    \end{enumerate}
\end{proof}

Lemmas \ref{lem:HC.des.geo} and \ref{lemma:an-FRS}  give us:
\begin{cor}
Assume that \Cref{thm:uncond.an.frs} holds for any smaller value of $n$.    
Then $p|_{\ug\smallsetminus \bfN_{insep}}$ is an-FRS.
\end{cor}
\subsection{Slices to nilpotent orbits}

The following Lemma is a version of \Cref{lem:slod.slice}.

\begin{lem}\label{lem:slod2.slice}    
    Assume $\chara(\F_\ell)>\frac n2$.
    Let $x\in \bfN(\F_\ell)$ be a non-regular element. 
 Then there exist
     \begin{itemize}
        \item a linear subspace $\bfM\subset \ug$\Rami{, and}
        \item a positive $\bG_m$ action on $\bfM$,
    \end{itemize}
    such that 
         \begin{enumerate}
        \item\label{lem:slod2.slice:2} the action map $\bfG\times (x+\bfM)\to \ug$ is a submersion.                \item\label{lem:slod2.slice:3}  for any nilpotent $y \in x+(\bfM(\bar\F_\ell)\smallsetminus 0)$  we have $\dim \bfG_{\bar\F_\ell} \cdot y > \dim \bfG_{\bar\F_\ell} \cdot x$ 
        \item \label{lem:slod2.slice:4} The Chevalley map $p|_{x+\bfM}$ intertwines the $\bG_m$ action on $x+\bfM$ (given by the identification $y\mapsto x+y$ between $\bfM$ and $x+\bfM$) with the $\bG_m$ action on $\uc$ given by $(\lambda \cdot f)(y):=\lambda^{n}f(\lambda^{-1} \cdot y)$.
        \item\label{lem:slod2.slice:6} If in addition $x$ is not subregular, then the sum of the exponents of the $\bG_m$ action on $\bfM$ is larger than $\frac{n(n+1)}{2}+1$.
        \item\label{lem:slod2.slice:7} $\uz \subset \bfM$ is $\bG_m$ invariant and the exponent of the action of $\bG_m$ on $\uz$ is $1$.
    \end{enumerate}    
\end{lem}
The proof is analogous to the proof of \Cref{lem:slod.slice}.
We will start with some preparations.
\begin{notation}
    Let $x$ be a nilpotent element in Jordan canonical form. Let
    $\bfL_x$ be the slice defined in Notation \ref{not:slice}.  
    Denote
    \begin{enumerate}
        \item $\bfM_x^0:=\{\{x_{ij}\}\in \bfL_x|x_{nn}=0\}.$\index{$\bfM_x^0$}
        \item $\bfM_x:=\bfM_x^0+\uz$\index{$\bfM_x$}.
    \end{enumerate} 
\end{notation}

\begin{proof}[Proof of \Cref{lem:slod2.slice}]

WLOG we may assume that $x$ is in a Jordan form and the size of the largest block is smaller than $\chara(\bF_\ell)$.
    
Take $\bfM:=\bfM_{x}$.
Define the action of $\bG_m$ on $\bfM$ by $$\lambda\star A:=\phi_{x}(\lambda) \cdot A.$$  where  $\phi_{x}$ is the morphism defined in Notation \ref{not:slice}\eqref{co-character.for. Jordan nilpotent}  and 
$\cdot$ is the action described in Notation \ref{not:slice}\eqref{action of GxG_m}.

It is easy to see that 
 this is a positive action.

 Conditions (\ref{lem:slod2.slice:2},\ref{lem:slod2.slice:3},\ref{lem:slod2.slice:6})
 are proven in the same way as in \Cref{lem:slod.slice}.

 Conditions (\ref{lem:slod2.slice:4},\ref{lem:slod2.slice:7}) are evident.
  
\end{proof}

\subsection{Proof of \Cref{thm:uncond.an.frs}}

We prove the theorem by induction $n$. Throughout the section we assume the validity of the result for smaller values of $n$.

The following is obvious:
\begin{lemma}
    $p|_{\ug^r}$ is an-FRS.
\end{lemma}

We obtain:
\begin{cor}\label[cor]{cor:HC.des} Let 
$\bfN_{insep,s}:=\bfN_{insep}\smallsetminus \ug^r$.      
Then $p|_{\ug\smallsetminus \bfN_{insep,s}}$ is an-FRS.
\end{cor}

\begin{remark}\label{rmk: Not too bad}
    Assume $\chara(\F_\ell)>\frac n2$. Then it is easy to see that $N_{insep,s}\subset N+\fz$. 
\end{remark}

The following is a global non-uniform version of \Cref{lem:FRS.hom.crit}. 

\begin{lemma}\label{lem:FRS.hom.crit2}
    Let $\gamma:\A^I\to \A^J$ be a polynomial map. 
    Assume that:
    \begin{enumerate}
        \item we are given 
        positive actions of $\bG_m$ on  $\A^I$ and $\A^J$  by $$s \cdot (x_1,\dots,x_I)=(s^{\lambda_1}x_1,\dots,s^{\lambda_I}x_I)$$ and $$s \cdot (y_1,\dots,y_J)=(s^{\mu_1}y_1,\dots,s^{\mu_J}y_J)$$ respectively. 
        \item         we are given an 
        action $\star$ of $\bG_a$ on $\bA^J$ s.t. for any $z\in \bG_a(\bar\F_\ell)$ and $x\in\bA^I(\bar\F_\ell)$ we have  $\gamma(x+z)=z\star\gamma(x)$. Here $\bG_a$ embeds into $\bA^I$ as the first coordinate.
        \item  $\star$ preserves the Haar measure on $F^I$.        
        \item $\sum_i  \lambda_i > 1+\sum_i  \mu_i$.
        \item $\lambda_1=1$
        \item $\gamma|_{\bA^I\smallsetminus \bA^1}:(\bA^I\smallsetminus \bA^1)\to \bA^J$ is an-FRS.
    \end{enumerate}
     Then 
     $\gamma$  is an-FRS.
\end{lemma}

\begin{proof}
WLOG we may assume that $\mu$ is an Haar measure on a ball $B\subset \F_\ell((t))^I$ around the origin.
Using the $\bG_m$ action we may assume WLOG that $B$  is the unit ball $\F_\ell[[t]]^I$.

Define another action of $\bG_m$ on $\A^I$ by: 
$$s * (x_1,\dots,x_I)=(x_1,s^{\lambda_2}x_2,\dots,s^{\lambda_I}x_I).$$ 
For $i\in \N$, define $C_i:=t^i*B\smallsetminus t^{i+1}*B$ and $\nu_i:=\mu 1_{C_i}$.

Let $S_i\subset \F_\ell[[t]]$ be the collection of polynomials of degree $\leq i-1$. We will also consider $S_i$ as a subset of $\F_\ell((t))^I$ using the  embedding $\bA^1\subset\bA^I$ corresponding to the first coordinate.
Note that $$C_i= \bigcup_{s\in S_i} \left(s+t^i\cdot C_0\right).$$

This implies $$\nu_i=\ell^{-i\sum_{j=1}^I\lambda_j } \sum_{s\in S_i} \left(sh_s(t^i\cdot \nu_0)\right),$$
where $sh_s$ stands for shift of a measure by $s$.

Denote $$g:=\frac{\gamma_*(\nu_0)}{\mu^{\bA^J,1}}$$
and let $$M=||g||_\infty.$$
We obtain:
\begin{align*}
    \gamma_*(\mu)=&\sum_{i=0}^\infty\gamma_*(\nu_i)=
    \sum_{i=0}^\infty
    \ell^{-i\sum_{j=1}^I\lambda_j } \sum_{s\in S_i} \gamma_*\left(sh_s(t^i\cdot \nu_0)\right)=    
    \sum_{i=0}^\infty
    \ell^{-i\sum_{j=1}^I\lambda_j } \sum_{s\in S_i} s\star \gamma_*\left(t^i\cdot \nu_0\right)=
    \\=&        
    \sum_{i=0}^\infty
    \ell^{-i\sum_{j=1}^I\lambda_j } \sum_{s\in S_i} s\star (t^i\cdot\gamma_*\left( \nu_0\right))
    =        
    \sum_{i=0}^\infty
    \ell^{-i\sum_{j=1}^I\lambda_j } \sum_{s\in S_i} s\star \left(t^i\cdot\left(\frac{\gamma_*\left( \nu_0\right)}{\mu^{\bA^J,1}} \mu^{\bA^J,1} \right)\right)=
    \\=&        
    \sum_{i=0}^\infty
    \ell^{-i\sum_{j=1}^I\lambda_j } \sum_{s\in S_i} s\star \left(\left(t^i\cdot g \right)
    \left(t^i\cdot\mu^{\bA^J,1} \right)\right)=
\\=&      
\sum_{i=0}^\infty
    \ell^{-i\sum_{j=1}^I\lambda_j + i\sum_{j=1}^J\mu_j } \sum_{s\in S_i} s\star \left(\left(t^i\cdot g \right)
    \mu^{\bA^J,1}\right)\leq
\\\leq&
\sum_{i=0}^\infty
    \ell^{-2i} \sum_{s\in S_i} s\star \left(M
    \mu^{\bA^J,1}\right)    
    =
    M\sum_{i=0}^\infty
    \ell^{-2i} \sum_{s\in S_i} \left(
    \mu^{\bA^J,1}\right)
    =   
    M\sum_{i=0}^\infty
    \ell^{-2i} |S_i| 
    \mu^{\bA^J,1}=
\\=&
    M\sum_{i=0}^\infty
    \ell^{-i}
    \mu^{\bA^J,1}\leq 2M\mu^{\bA^J,1}.
\end{align*}
This implies the assertion.
\end{proof}

\begin{lem}\label{lem:zplusO}
Assume $\chara \F_\ell>n/2$.
    Let $\bfO\subset \bfN$ be a non-regular nilpotent orbit. Then for any field $E/ \F_\ell$ we have $$(\uz+ \bfO)(E)=\uz(E)+ \bfO(E).$$
\end{lem}
\begin{proof}
    Let $x\in (\uz+ \bfO)(E)$. First note that it is non-regular as an element in $x\in \ug(\bar E)$. This implies that it is non-regular as an element in $\ug(E)$. Let $f$ be its characteristic polynomial and $g$ be its minimal polynomial. Let $h=f/g$. Note that both $h,g$ are defined over 
    $E$. Also, for some $\lambda\in \bar E$ and $k\in \N$ we can write: $g=(x-\lambda)^k$ and $h=(x-\lambda)^{n-k}$. Let $m=\min(k,n-k)$. We get that $(x-\lambda)^m$ is defined over $E$ and that $0<m<p$. This implies that $\lambda\in E$, which implies the assertion.
\end{proof}

\Cref{lemma:an-FRS} implies the following:
\begin{cor}
\label[cor]{cor:an-FRS.alg} 
 Let $\gamma:\bfX \to \bfY$ be a morphism of algebraic varieties. 
    Let $\delta:\tilde \bfX \to \bfX$ be a submersion.
    Then:
    \begin{enumerate}
        \item if $\gamma$ is an-FRS then so is $\gamma \circ \delta$.
        \item If $\delta(\tilde\bfX(\F_{\ell}((t))))=\bfX(\F_{\ell}((t)))$  and  $\gamma \circ \delta$ is  an-FRS  then so is $\gamma$.
    \end{enumerate}  
    
\end{cor}

The classical argument for the bounds of Cauchy and Lagrange gives:
\begin{lem}\label{lem: roots Cauchy-Lagrange}
    Let $f\in \uc(O_{\bar F})$. Then any  root $\lambda\in \bar F$ of $f$ satisfies $\lambda\in O_{\bar F}$.
\end{lem}
\begin{lem}\label{lem:density}
    Let $f \in O_{\bar F}[x] $ be a  monic polynomial in one variable with coefficients in $O_{\bar F}$. Then $$\int_{O_F} val(f(x))dx\leq \frac{\deg(f)}{\ell-1}.$$
\end{lem}
\begin{proof}
$ $
    \begin{enumerate}[{Case} 1:]
        \item $f(x)=x$\\
\begin{align*}
\int_{O_F} val(x)dx&=\sum_{m=0}^{\infty}\int_{|x| = \ell^{-m}} val(x)=\sum_{m=0}^{\infty} m Vol(\{x\in O_F : |x|= \ell^{-m}\}) dx=C\sum_{m=0}^{\infty} m \ell^{-m}
\end{align*}
with $C=vol(\{x \in O_{F}: |x|=1\})=1-\frac{1}{\ell}$,
but 
\begin{align*}
\sum_{m=0}^{\infty} m \ell^{-m}&=\sum_{m=1}^{\infty}  m\ell^{-m}=
\left. x\frac{d}{dx} \left (\sum_{m=0}^{\infty}  x^{m}\right)\right|_{x=\frac{1}{\ell}}=
\left.x\frac{d}{dx} \left (\frac{1}{1-x}\right)\right|_{x=\frac{1}{\ell}}=
\\&=
\left. \left(\frac{x}{(1-x)^2}\right)\right|_{x=\frac{1}{\ell}}=
\frac{1}{\ell(1-\frac{1}{\ell})^2}=\frac{\ell}{(\ell-1)^2}
\end{align*}
To sum up, 
$$
\int_{O_F} val(x)dx=\frac{\ell}{(\ell-1)^2}\left(1-\frac{1}{\ell}\right)=\frac{1}{\ell-1}=\frac{\deg(f)}{\ell-1}
$$
        
      \item $f(x)=(x-\lambda)$ for $\lambda\in O_F$\\
      Follows from previous case.
        \item $f(x)=(x-\lambda)$ for $\lambda\in O_{\bar F}$\\
Let $\lambda_0\in O_F$ be s.t. $$|\lambda-\lambda_0|=\min_{\mu\in O_F} |\lambda-\mu|.$$
It is easy to see that $val(x-\lambda) \leq val(x-\lambda_0)$ for any $x \in O_F$ and the assertion follows from previous case.      
        \item $f(x)=(x-\lambda)^k$ for $\lambda\in O_{\bar F}$\\
         Follows from previous case. 
        \item General case\\
         Follows from previous case and \Cref{lem: roots Cauchy-Lagrange}.
    \end{enumerate}
\end{proof}

\begin{lem}\label{lem:add.fin}
    The addition map $add:\bfN \times \uz\to \ug$ is finite. 
\end{lem}
\begin{proof}
    This follows from the following Cartesian square:
    $$
\begin{tikzcd}
\bfN \times \uz \ar[r,"add"] \ar[d] \arrow[dr, phantom, "\square"] & \ug \ar[d, "p"] \\
\uz \ar[r, "p_\uz"'] & \uc
\end{tikzcd}
$$

\end{proof}
\begin{lem}\label{lem:mult.push}
Let $m:\A^2\to \A^1$ be the multiplication map.
We have 
 $$
m_*(\mu^{\A^2,1}_0)=r(z)\mu^{\A^1,1}_0,
$$
where $$r(z)=\frac{\ell-1}{\ell}(val(z)+1).$$
\end{lem}
\begin{proof}
\begin{align*}
    m_*(\mu^{\A^2,1}_0)(t^{r}O_F\smallsetminus t^{r+1}O_F)=&\mu^{\A^2,1}_0(m^{-1}((t^{r}O_F\smallsetminus t^{r+1}O_F)))=\\ &\sum_{i=0}^r \mu^{\A^2,1}_0((t^{i}O_F\smallsetminus t^{i+1}O_F))\times (t^{r-i}O_F\smallsetminus t^{r-i+1}O_F))=
    \\ &=\sum_{i=0}^r \ell^{-r}\mu^{\A^1,1}_0(O_F\smallsetminus tO_F)^2=\frac{(\ell-1)^2}{\ell^2}(r+1) \ell^{-r}
\end{align*}
Using the action of $O_F^{\times}$, this implies the assertion.
\end{proof}

\begin{proof}[Proof of \Cref{thm:uncond.an.frs}]
Let $\bfN^{s}\subset \ug$ be the cone of non-regular nilpotent elements.
Enumerate  the nilpotent orbits $\{0\}=\bfO_1,\dots \bfO_m$ s.t. $\dim \bfO_i\leq \dim \bfO_j$ for any $i<j$. Let $\bfN_i=\bigcup_{j=1}^i \bfO_j$. Note that $\bfN_i$ is closed and $\bfN_0=\emptyset$. Therefore, by \Cref{lem:add.fin}, $\bfN_i+\uz$ are closed (and $\bfN_0+\uz=\emptyset$). 
We will prove  by down going induction on $i$ that for any $i\geq 0$ the map
$p|_{\ug \smallsetminus (\bfN_i+\uz)}$ is an-FRS.

The base of the induction $i=m$ follows from \Cref{cor:HC.des}.
For the induction step, we assume the statement holds for $\bfN_{i+1}$ and prove it for $\bfN_{i}$. Let $\bfU=\ug\smallsetminus (\bfN_{i+1}+\uz)$.
Choose a  nilpotent element $x\in \bfO_{i+1}(\F_\ell)$.  Let $\bfM$ be the linear space given by \Cref{lem:slod2.slice} when applied to 
$x$. 
\begin{enumerate}[Step 1.]
    \item Reduction to $\bfG\times \bfM$.\\
Consider the map $$\delta:(\bfG\times \bfM) \sqcup \bfU\to 
\ug\smallsetminus (\bfN_{i}+\uz)$$ given on $\bfG\times \bfM$ by $\delta(g,l):=g\cdot (x+l)$ and on $\bfU$ by the embedding $\bfU\subset \ug\smallsetminus (\bfN_{i}+\uz)$. By \Cref{lem:slod2.slice}\eqref{lem:slod2.slice:2} it is submersive. Also, it is  onto on the level of points over any field. Indeed, for any extension $E/\F_\ell$  we have 
$(\bfO_{i+1}+\uz)(E)=\bfO_{i+1}(E)+\uz(E)=\bfG(E)\cdot (x+\uz)$ where the first equality is by \Cref{lem:zplusO}. Thus,
\begin{multline*}
    (\ug\smallsetminus (\bfN_i+\uz))(E)=(\bfO_{i+1}+\uz)(E) \cup \bfU(E)=\\= (\bfG(E)\cdot (x+\uz) )\cup \bfU(E)\subset \delta(((\bfG\times \bfM) \sqcup \bfU)(E)).
\end{multline*}
Thus, by Corollary  \ref{cor:an-FRS.alg}, it is enough to show that $p\circ \delta$ is an-FRS.
Let $\delta':=\delta|_{\bfG\times \bfM}$.
Notice that $p|_{\bf U}$ is an-FRS by the induction hypothesis.
Therefore it is enough to show that $p\circ \delta'$ is an-FRS.

Notice that 
$\delta' (\bfG\times (\bfM\smallsetminus \uz)) \subset \bfU $ by \Cref{lem:slod2.slice}\eqref{lem:slod2.slice:3}. So,  by  Corollary  \ref{cor:an-FRS.alg}  we deduce 
that 
$p\circ \delta'|_{\bfG\times (\bfM\smallsetminus \uz) }$ is an-FRS. 
\item Reduction to $\bfM$.\\
We can factor the map $p\circ \delta'$ as $p|_{(x+\bfM)} \circ sh_x \circ pr_{\bfM}$, where $sh_x:\bfM\to x+\bfM$ is the shift map,  and $pr_{\bfM}:\bfG\times \bfM\to \bfM$ is the projection. So, by \Cref{cor:an-FRS.alg} it is enough to show that $p|_{x+\bfM}\circ sh_x:\bfM \to \fc$ is an-FRS.

Also, by \Cref{cor:an-FRS.alg} we deduce 
that 
$p|_{x+(\bfM\smallsetminus \uz)} \circ sh_x$ is an-FRS.

\item Proof that $p|_{x+\bfM}\circ sh_x:\bfM \to \fc$ is an-FRS when $x$ is not subregular. \\
    The assertion follows now from \Cref{lem:FRS.hom.crit2} and the condition on $\bfM$ given by \Cref{lem:slod2.slice}(\ref{lem:slod2.slice:6},\ref{lem:slod2.slice:7}).

\item\label{step:4.pf.E} Proof that $p|_{x+\bfM}\circ sh_x:\bfM \to \fc$ is an-FRS when $x$ is subregular.\\ 
In this case \Cref{lem:FRS.hom.crit2}  is not applicable.
So, we provide a direct argument.
Using the action of  $\bG_m$ on $\bfM$ it is  enough to show that $p_*(\mu^{x+ \bfM,1}_0)$ has bounded density.

Without loss of generality we assume $x$ is in Jordan form of type $(n-1,1)$. More precisely $x=J_{n-1}(0) \oplus J_{1}(0)$.


Note that, until now, we only use the fact that $\bfM$ satisfies the conditions of \Cref{lem:slod2.slice}. 

So we can choose any such $\bfM$. Choose $\bfM$ as in the proof of \Cref{lem:slod2.slice}. Explicitly, 
 $$\bfM=\{\ c(f)+(\alpha-1) e_{n-1,n} +z|z\in \uz; f\in \uc; \alpha \in \A^1\}.$$
Here  $c(f)$ is the companion matrix corresponding to $f\in \uc$ where the coefficients are located in the first row.

Notice that $\bf M$ is the space of matrices of the form
$$\begin{pmatrix}
    \ast & 1 \\
    \ast & z & 1\\
    \ast & 0 & z & 1\\
    \cdots & \cdots & \cdots & \cdots & \cdots \\
    \ast & 0 & \cdots & 0 & z & 1\\
    \ast & 0 & \cdots & 0 & 0 & z & \alp\\
    \ast & 0 & \cdots & 0 & 0 & 0 & z 
\end{pmatrix}$$

Denote $$\bfM_1=\{\ c(f)+(\alpha-1) e_{n-1,n} | f\in \uc; \alpha \in \A^1\}.$$

It is easy to see that for any $c(f)+(\alpha-1) e_{n-1,n} \in \bfM_1$ we have 
$$p(c(f)+(\alpha-1) e_{n-1,n})=f-f(0)+\alpha f(0).$$  



Let us write $\uc=\A^{n-1} \oplus \A^1$
where  $\A^{n-1}$ 
represents the leading $n-1$ coefficients 
of an element in $\uc$ and $\A^1$ corresponds to the constant term.
Now we have
$$p_*(\mu_{0}^{x+ \bfM_1,1})=\mu^{\A^{n-1},1}_{0}\boxtimes m_*(\mu_{\A^2}),$$
where $m:\A^2\to \A^1$ is the multiplication map.


This together with \Cref{lem:mult.push}  implies that
$$p_*(\mu^{x+ \bfM_1,1}_0)=\Phi\mu^{\uc,1}_{0},$$
where $\Phi\in L^1(\fc)$ is given by $$\Phi(g)=\frac{\ell-1}{\ell}(1+val(g(0))),$$ where $g \in \uc.$


From this we deduce 
 $p_*(\mu^{x+ \bfM,1}_0)=h\mu^{\uc,1}_{0}$, where $h\in L^1(\fc)$ is given by 
 $$h(g)=\int_{O_F} \Phi(sh_{z}(g)) dz=\frac{\ell-1}{\ell}\int_{O_F} (val(g(z))+1) dz.$$
 Here $sh_z$ is the shift by $z$ of a polynomial.

The assertion follows now from \Cref{lem:density} as
$$
 h(g)=\frac{\ell-1}{\ell}\int_{O_F} (val(g(z))+1) dz \leq  \frac{1}{\ell}deg(g)+\frac{\ell-1}{\ell} \leq \frac{n}{\ell}+1
$$
is a bounded function.
\end{enumerate}
\end{proof}


\begingroup
  \let\clearpage\relax
  \let\cleardoublepage\relax 
  \printindex
\endgroup

\bibliographystyle{alpha}
\bibliography{Ramibib}

\end{document}